\documentclass{amsart}
\usepackage{amscd, amsmath,amssymb,amsfonts}
\usepackage[cmtip, all]{xy}

\newtheorem{thm}{Theorem}[section]
\newtheorem{prop}[thm]{Proposition}
\newtheorem{lem}[thm]{Lemma}
\newtheorem{cor}[thm]{Corollary}

\renewcommand{\theclaim}{\kern-3pt}

\theoremstyle{definition}
\newtheorem{definition}[thm]{Definition}

\theoremstyle{remark}
\newtheorem{rem}[thm]{Remark}
\newtheorem{rems}[thm]{Remarks}
\newtheorem{ex}[thm]{Example}

\numberwithin{equation}{section}

\newcommand{\sC}{{\mathcal C}}

\newcommand{\sE}{{\mathcal E}}
\newcommand{\sF}{{\mathcal F}}

\newcommand{\sH}{{\mathcal H}}
\newcommand{\sI}{{\mathcal I}}

\newcommand{\sK}{{\mathcal K}}
\newcommand{\sL}{{\mathcal L}}
\newcommand{\sM}{{\mathcal M}}

\newcommand{\sO}{{\mathcal O}}

\newcommand{\sS}{{\mathcal S}}

\newcommand{\sV}{{\mathcal V}}

\newcommand{\sX}{{\mathcal X}}

\newcommand{\A}{{\mathbb A}}

\newcommand{\G}{{\mathbb G}}

\renewcommand{\P}{{\mathbb P}}

\newcommand{\mS}{{\mathbb S}}

\newcommand{\Z}{{\mathbb Z}}

\renewcommand{\L}{{\mathbb L}}

\renewcommand{\phi}{\varphi}

\newcommand{\CH}{{\rm CH}}

\newcommand{\red}{{\rm red}}
\newcommand{\codim}{{\rm codim}}

\newcommand{\Div}{{\rm Div}}
\newcommand{\Hom}{{\rm Hom}}

\newcommand{\im}{{\rm im}}
\newcommand{\Spec}{{\rm Spec \,}}

\newcommand{\supp}{{\rm supp}\,}
\newcommand{\0}{\emptyset}
\newcommand{\sHom}{{\mathcal{H}{om}}}

\newcommand{\id}{{\operatorname{id}}}
\newcommand{\Zar}{{\text{\rm Zar}}} 
\newcommand{\Ord}{{\mathbf{Ord}}}

\newcommand{\Sch}{{\operatorname{\mathbf{Sch}}}}

\newcommand{\holim}{\mathop{{\rm holim}}}
\newcommand{\op}{{\text{\rm op}}}

\newcommand{\del}{\partial}
\newcommand{\fib}{{\operatorname{\rm fib}}}
\renewcommand{\max}{{\operatorname{\rm max}}}

\newcommand{\Spt}{{\mathbf{Spt}}}
\newcommand{\Spc}{{\mathbf{Spc}}}
\newcommand{\Sm}{{\mathbf{Sm}}}

\newcommand{\hocolim}{\mathop{{\rm hocolim}}}
\newcommand{\colim}{\operatornamewithlimits{\varinjlim}}

\newcommand{\GL}{{\operatorname{\rm GL}}}

\newcommand{\SH}{{\operatorname{\sS\sH}}}

\newcommand{\eff}{{\mathop{eff}}}

\newcommand{\ds}{{/\kern-3pt/}}

\newcommand{\Th}{{\mathop{\rm{Th}}}}

\newcommand{\SP}{{\mathbf{SP}}}

\newcommand{\MGL}{{\operatorname{MGL}}}

\newcommand{\lci}{\text{l.\,c.\,i.\,}}
\newcommand{\Tor}{{\operatorname{Tor}}}

\newcommand{\GrAb}{\mathbf{GrAb}}

\newcommand{\ch}{\rm{ch}}

\begin{document}
\title{Connective Algebraic $K$-theory}
\author{Shouxin Dai}
\address{
}
\address{
Mathematics Institute\\
Chinese Academy of Sciences\\
55 Zhongguancun East Road\\
100190 Beijing\\
China}
\email{shouxindai@163.com}
\author{Marc Levine}
\address{
Universit\"at Duisburg-Essen\\
Fakult\"at Mathematik\\
Thea-Leymann-Stra{\ss}e 9\\
45127 Essen\\
Germany}
\email{marc.levine@uni-due.de}


 
 \thanks{The 
author M.L.  thanks the Humboldt Foundation for support through the Humboldt Professorship. The author S.D. warmly thanks continuous support  and encouragement of M. Levine throughout the years and expresses his gratitude to X. Sun for support in many ways. He is indebted to A. Merkurjev for inspirational communications.}

\begin{abstract}
We examine the theory of connective algebraic $K$-theory, $\sC\sK$, defined by taking the -1 connective cover of algebraic $K$-theory with respect to Voevodsky's slice tower in the motivic stable homotopy category. We extend $\sC\sK$ to a bi-graded oriented duality theory $(\sC\sK'_{*,*}, \sC\sK^{*,*})$ in case the base scheme is the spectrum of a field $k$ of characteristic zero. The homology theory $\sC\sK'_{*,*}$ may be viewed as connective algebraic $G$-theory. We identify $\sC\sK'_{2n,n}(X)$  for $X$ a finite type $k$-scheme with the image of $K_0(\sM_{(n)}(X))$ in $K_0(\sM_{(n+1)}(X))$,  where $\sM_{(n)}(X)$ is the abelian category of coherent sheaves on $X$ with support in dimension at most $n$; this agrees with the (2n,n) part of the theory of connective algebraic $K$-theory defined by Cai. We also show that the classifying map from algebraic cobordism identifies $\sC\sK'_{2*,*}$ with the universal oriented Borel-Morel homology theory $\Omega_*^{CK}:=\Omega_*\otimes_\L\Z[\beta]$ having formal group law $u+v-\beta uv$ with coefficient ring $\Z[\beta]$. As an application, we show that every pure dimension $d$ finite type $k$ scheme has a well-defined fundamental class $[X] _{CK}$ in $\Omega_d^{CK}(X)$, and this fundamental class is functorial with respect to pull-back for lci morphisms. Furthermore, the fundamental class $[X] _{CK}$ maps to the usual fundamental classes $[X] _{Chow}$, reap. $[X] _{K}$ under the natural maps
\[
\Omega_*^{CK}\to K_0[\beta,\beta^{-1}];\quad \Omega_*^{CK}\to\CH_*
\]
given by inverting $\beta$, resp. moding out by $\beta$.
\end{abstract}

\maketitle

\tableofcontents

\section*{Introduction}
In topology, the theory of {\em connective $K$-theory} is represented by the -1 connected cover $ku$ of the topological $K$-theory spectrum $KU$. In the setting of motivic stable homotopy theory over a base-scheme $S$, Voevodsky  \cite{Voevodsky} has constructed the algebraic analog of $KU$, namely the algebraic $K$-theory spectrum $\sK_S$, which represents Quillen's algebraic $K$-theory on the category of smooth $S$-schemes, assuming that $S$ itself is a regular scheme (see also \cite{PaninPimenovRoendigs2}), in that there are natural isomorphisms $\sK_S^{a,b}(X)\cong K_{2b-a}(X)$ for $X$ a smooth finite type $S$-scheme.

There are a number of possible notions of connectivity in the motivic stable homotopy category over $S$, $\SH(S)$, but one that has proved quite useful is given by using $T$-connectivity ($T=\A^1/\A^1\setminus\{0\}$), in the sense of the tower of localizing  subcategories
\[
\cdots \subset\Sigma^{n+1}_T\SH^\eff(S)\subset \Sigma^n_T\SH^\eff(S)\subset\ldots\subset \SH(S),
\]
where $\SH^\eff(S)$ is the localizing subcategory of $\SH(S)$ generated by the $T$-sus\-pen\-sion spectra of smooth $S$-schemes. The associated truncation functors give rise to Voevodsky's {\em slice tower}
\[
\dots \to f_{n+1}\sE\to f_n\sE\to\dots\to \sE
\]
for each $\sE\in \SH(S)$, with $f_n\sE\to \sE$ the universal morphism from an object in $\Sigma^n_T\SH^\eff(S)$ to $\sE$; one could call this the ``$T$-$(n-1)$-connected cover" of $\sE$. It is thus natural to define connective algebraic $K$-theory as the bi-graded cohomology theory represented by $f_0\sK_S$. 

In this paper we study connective algebraic $K$-theory, $\sC\sK^{*,*}$ and its associated oriented homology theory,  $\sC\sK_{*,*}'$, this latter for $S=\Spec k$, $k$ a characteristic zero field (see however remark~\ref{rem:ConnGThy}). The  oriented homology theory $\sC\sK_{*,*}'$ is the connective analog of $G$-theory, that is, the $K$-theory of coherent sheaves rather than vector bundles. The canonical map $\sC\sK^{*,*}\to \sK^{*,*}$ induces the map $\sC\sK_{*,*}'\to G_{*,*}$ (where    $G_{a,b}(X):=G_{a-2b}(X)$) and we elucidate here how the connective versions refine the non-connective ones.

Cai \cite{Cai} has defined a bi-graded oriented Borel-Moore homology theory, which he calls connective algebraic $K$-theory, by using the Quillen-Gersten spectral sequence. Concretely, he defines the group $CK^{a,b}(X)$ as the image of $K_{2b-a}(\sM^{(a)}(X))$ in $K_{2b-a}(\sM^{(a-1)}(X))$, where $\sM^{(a)}(X)$ is the category of coherent sheaves on $X$ supported in codimension at least $a$ (this is for $X$ smooth or at least equi-dimensional over a field $k$; in the general case, one indexes using dimension giving the theory $CK'_{*,*}$). Cai verifies the properties of an oriented Borel-Moore homology theory for  $CK^{*,*}$. It turns out this Cai's theory does not in general agree with the one given by motivic homotopy theory, but it does agree in the portion corresponding to $K_0$ or $G_0$ (see theorem~\ref{thm:ConnK0} and remark~\ref{rem:Cai}).

Besides comparison results of this type and other structural properties of connective $K$-theory, our main result is a comparison with algebraic cobordism. There is a canonical natural transformation
\[
\Omega_*(X)\otimes_\L\Z[\beta]\to \sC\sK_{2*,*}'(X)
\]
which, in the case of characteristic zero, is an isomorphism for all quasi-projective $k$-schemes $X$ (see theorem~\ref{thm:ClassIso}). This allows us to use the natural fundamental classes in $G$-theory, namely, the class of the structure sheaf, to define a fundamental class $[X]\in
\Omega_d(X)\otimes_\L\Z[\beta]$ for $X$ of pure dimension $d$ over $k$, which is functorial with respect to pull-back by \lci morphisms. Note that such fundamental classes in $\Omega_d(X)$ do not exist in general \cite{LevineFundClass}.

We let $\Spc$ and $\Spc_\bullet$ denote the categories of simplicial sets and pointed simplicial sets, respectively, with homotopy categories $\sH$ and $\sH_\bullet$. $\Spt$ the category of spectra (for the usual suspension operator $\Sigma:=(-)\wedge S^1$) and $\SH$ the stable homotopy category.

For a scheme $S$, $\Sch/S$ will denote the category of quasi-projective schemes over $S$, $\Sm/S$ the full subcategory of smooth quasi-projective schemes over $S$.  $\Spc(S)$, and $\Spc_\bullet(S)$ the categories of pre sheaves on $\Sm/S$ with values in $\Spc$, $\Spc_\bullet$, $\Spt_{S^1}(S)$ the category of $S^1$-spectra over $S$, this being the category of presheaves of spectra on $\Sm/S$. We let $\Spt_T(S)$ denote the category of $T$-spectra in $\Spc_\bullet(S)$, with $T:=\A^1/\A^1\setminus\{0\}$, and $\Spt^\Sigma_T(S)$ the category of symmetric $T$-spectra. We have as well the category of $S^1-\G_m$ bi-spectra objects in $\Spc_\bullet(S)$, denoted $\Spt_{s,g}(S)$.

The categories $\Spc(S)$,  $\Spc_\bullet(S)$, $\Spt_{S^1}(S)$, $\Spt_T(S)$, $\Spt_{s,g}(S)$ and $\Spt^\Sigma_T(S)$ all have so-called {\em motivic} model structures (the original source for the unstable theory is  \cite{MorelVoev}, see also \cite{Motivic, GoerssJardine} for a compact description. For the stable theory, we refer the reader to \cite{Jardine2}), with homotopy categories denoted $\sH(S)$, $\sH_\bullet(S)$, $\SH_{S^1}(S)$, $\SH(S)$, $\SH_{s,g}(S)$ and $\SH^\Sigma(S)$, respectively. The categories $\SH(S)$, $\SH_{s,g}(S)$ and $\SH^\Sigma(S)$ are equivalent triangulated tensor categories; the tensor structure is induced by a  symmetric monodical structure on  $\Spt^\Sigma_T(S)$. There is also a symmetric spectra version of $\Spt_{S^1}(S)$ making $\SH_{S^1}(S)$ a tensor triangulated category.

We denote by $\G_m$ the pointed $S$-scheme $(\A^1_S-0_S,1_S)$. $\P^1_*$ will denote  the pointed $S$-scheme $\P^1_S$, with base-point $1_S$.  

$\Ord$ is the category of finite ordered sets, we let $[n]\in\Ord$ denote the set $\{0,\ldots, n\}$ with the standard ordering. We let $\L$ denote the Lazard ring, that is, the coefficient ring of the universal rank one commutative formal group law $F_\L\in \L[[u,v]]$. We let $\L^*$ denote $\L$ with the grading determined by $\deg a_{ij}=1-i-j$ if $F_\L(u,v)=u+v+\sum_{i,j\ge1}a_{ij}u^iv^j$ and let $\L_*$ denote $\L$ with the opposite grading $\L_n:=\L^{-n}$.

\section{$K$-theory and connective $K$-theory}
We work in the motivic stable homotopy category $\SH(S)$ over a {\em regular} base-scheme $S$; we will rather quickly pass to the case $S=\Spec k$ for a field $k$.

We have the Tate-Postnikov tower
\[
\ldots\to f_{n+1}\to f_n\to\ldots\to \id
\]
of endofunctors of $\SH(S)$ associated to the inclusions of full localizing subcategories
\[
\cdots\subset  \Sigma_T^{n+1}\SH^\eff(S)\subset \Sigma_T^{n}\SH^\eff(S)\subset\cdots\subset \SH(S).
\]
with $\Sigma_T^{n}\SH^\eff(S)$ the full localizing subcategory of $\SH(S)$ generated by the objects $\Sigma^m_TX_+$ for $X\in \Sm/S$, $m\ge n$. Letting $i_n:\Sigma_T^{n}\SH^\eff(S)\to \SH(S)$ be the inclusion, $i_n$ admits the right adjoint $r_n:\SH(S)\to \Sigma_T^{n}\SH^\eff(S)$ and $f_n$ is by definition the composition $i_n\circ r_n$; the natural transformation $f_n\to \id$ is the co-unit of the adjunction. There are functors $s_n$, $n\in \Z$ which give a distinguished triangle $f_{n+1}\to f_n\to s_n\to f_{n+1}[1]$; one can show that these distinguished triangles are unique up to unique natural isomorphism. In addition, we have the canonical isomorphism 
\[
f_{n+1}\circ\Sigma_T\cong \Sigma_T\circ f_n.
\]
For details, we refer the reader to \cite{LevineHC, VoevSlice1, VoevSlice1bis}.

\begin{rem} \label{rem:Ring} By \cite[proposition 5.3]{GRSO}, if $\sE$ is given a lifting to an $E_\infty$ ring object in $\Spt^\Sigma_T(S)$, then $f_0\sE$ has a canonical lifting to an $E_\infty$ ring object in $\Spt^\Sigma_T(S)$. In particular,  $f_0\sE$ is itself a commutative unital monoid in $\SH(S)$ and thus defines a bi-graded cohomology theory on $\Sm/S$ in the sense of Panin \cite[definition 2.0.1]{Panin2}.
\end{rem}

Following \cite[lemma 2.5, theorem 4.1]{RSO} we have a commutative unital monoid object $\sK_S$ of $\Spt^\Sigma_T(S)$ representing algebraic $K$-theory in  $\SH(S)$,  in the sense that, for each $X\in\Sm/S$ and open subscheme $U$, there is an isomorphism $\sHom(\Sigma_T^\infty X/U, \sK_S)\cong K(X\text{ on }X\setminus U)$ in $\SH$, natural in pairs $(X,U)$. Here $\sHom(-,-)$ is the $\Spt$-valued enriched Hom and $K(X\text{ on }X\setminus U)$ is the algebraic $K$-theory spectrum of the category of perfect complexes on $X$ with support on $X\setminus U$. Here we use our standing assumption that $S$ is regular; for general $S$, $\sK_S$ represents  homotopy invariant $K$-theory $KH$. 

To fix the conventions, the image of $\sK_S$ in $\SH(S)$ under the forgetful functor is isomorphic in $\SH(S)$ to a $T$-spectrum of the form $(\bar{\sK}_S, \bar{\sK}_S, \ldots)$, with $\bar{\sK}_S$ a commutative unital monoid in $\sH_\bullet(S)$ representing $K$-theory on $\Sm/S$. The structure map $\bar{\sK}_S\wedge T\to \bar{\sK}_S$ is isomorphic in $\sH_\bullet(S)$ to the composition
\[
\bar{\sK}_S\wedge T\xrightarrow{\id\wedge\sigma}\bar{\sK}_S\wedge \P^1\xrightarrow{\id\wedge\gamma} \bar{\sK}_S\wedge \bar{\sK}_S\xrightarrow{\mu}\bar{\sK}_S
\]
where $\mu$ is the multiplication , $\rho:T\to \P^1_*$ is the standard isomorphism and $\gamma:\P^1_*\to \bar{\sK}_S$ represents the class $[\sO]-[\sO(-1)]$ in $K_0(\P^1_*)$. This agrees with the convention in \cite{RSO}, and is the negative of the convention used in \cite{PaninPimenovRoendigs2}. We will often write $\sK_S$ for the image of $\sK_S$ in $\SH(S)$, with the context determining the meaning. 

\begin{rem}\label{rem:Conventions} 1. The fact that the shift operator leaves the $T$-spectrum $(\bar{\sK}_S, \bar{\sK}_S, \ldots)$ unchanged gives us the {\em Bott periodicity isomorphism} $\Sigma_T\sK_S\cong \sK_S$ in $\SH(S)$. We note that the map $\gamma:\P^1_*\to \bar{\sK}_S$ thus gives rise to an element $[\gamma]\in \sK^{2,1}(\P^1_*)$ which corresponds to the unit $1\in \sK^{0,0}(S)$ under the suspension isomorphism $\sK^{2,1}(\P^1_*)\cong\sK^{0,0}(S)$.\\
\\
2. Let $t_\sK\in \sK_S^{1,1}(\G_m)$ be the element corresponding to $1\in \sK^{0,0}(S)$ under the suspension isomorphism $\sK^{1,1}(\G_m)\cong\sK^{0,0}(S)$. Thus $\Sigma_{S^1}t_\sK=[\gamma]$. Under the isomorphisms $K_1(\G_m)\cong \sK^{1,1}(\G_m)$, $K_0(\P^1_*)\cong \sK^{2,1}(\P^1_*)$, the isomorphism $\sK^{1,1}(\G_m)\cong \sK^{2,1}(\Sigma_{S^1}\G_m)\cong  \sK^{2,1}(\P^1_*)$ is identified with the boundary map in the Mayer-Vietoris sequence
\[
\cdots\to K_1(\A^1)\oplus K_1(\A^1)\to K_1(\A^1\setminus \{0\})\xrightarrow{\partial}K_0(\P^1)\to\cdots
\]
Let $[t]\in K_1(\A^1\setminus \{0\})$ be the image of the canonical coordinate $t$ on $\A^1\setminus \{0\}=\Spec \sO_S[t,t^{-1}]$. Then $\partial[t]=1-[O(-1)]$ (at least up to a universal sign, see \cite[\S 7, lemma 5.16]{QuillenKThyI}), and hence $t_\sK=[t^{\pm1}]$.  
\end{rem}

\begin{definition} Let $\sC\sK_S$ be the object $f_0\sK_S$ of  $\SH(S)$, with canonical map $\rho:\sC\sK_S\to \sK_S$. {\em Connective algebraic $K$-theory} over $S$ is the bi-graded cohomology theory on $\Sm/S$ represented by  $\sC\sK_S$. We call $\sC\sK_S$ the connective algebraic $K$-theory $T$-spectrum (over $S$). 
\end{definition}

\begin{rem} The fact that $\sC\sK_S$ lifts to an $E_\infty$-ring object in $\Spt^\Sigma_T(S)$ is noted in \cite[\S 6.2]{GRSO}.
\end{rem}

\begin{lem}\label{lem:0slice} Suppose $s_0\sK_S$ satisfies $(s_0\sK_S)^{a,0}(X)=0$ for $a\neq0$ and for all $X\in\Sm/S$.  For $X\in\Sm/S$, consider the natural map $\rho^{a,1}(X):\sC\sK_S^{a,1}(X)\to \sK_S^{a,1}(X)$.  Then $\rho^{2,1}$ is injective, $\rho^{3,1}$ is surjective and $\rho^{a,1}$ is an isomorphism for $a\neq 2,3$.
\end{lem}

\begin{proof} By Bott periodicity, we have $\Sigma_Tf_0\sK_S=f_1\Sigma_T\sK_S=f_1\sK_S$. Thus $\sC\sK^{a,1}(X)=(f_1\sK)^{a-2,0}(X)$. Since $\Sigma^\infty_TX_+$ is in $\SH^\eff(S)$, we have $\sK^{*,0}(X)=(f_0\sK_S)^{*,0}(X)$,  and we thus  have a long exact sequence
\[ 
\ldots\to(s_0\sK_S)^{a-3,0}(X)\to \sC\sK^{a,1}(X)\xrightarrow{\rho} \sK^{a,1}(X)\to (s_0\sK_S)^{a-2,0}(X)\to\ldots .
\]
Our assumption $(s_0\sK_S)^{a,0}(X)=0$ for $a\neq0$ yields the desired result.
\end{proof}

\begin{rem}\label{rem:SmoothCompatibiilty} For $S=\Spec k$, $k$ a perfect field, we know by \cite[theorem 6.6]{VoevSlice2} \cite[theorem 10.5.1]{LevineHC} that $s_0\sK_S$ is isomorphic in $\SH(k)$ to the $T$-spectrum representing motivic cohomology, $H\Z$. We also know that $H\Z^{a,0}(X)=H^a(X_\Zar,\Z)$ for $X\in\Sm/k$, hence the assumption in lemma~\ref{lem:0slice} is satisfied. In addition, the map $\sK_k^{2,1}(X)=K_0(X)\to H\Z^{0,0}(X)$ is just the rank homomorphism, hence surjective. As $\sK_S^{a,1}(X)=K_{2-a}(X)=0$ for $a>2$, we see that $\sC\sK^{a,1}(X)=0$ for $a>2$ as well. 

The functors $f_n$ and $s_n$ are compatible with pull-back by smooth morphisms \cite[theorem 2.12, remark 2.13]{PelaezFunctSlice}, and for $f:T\to S$ smooth, we have $f^*\sK_S\cong \sK_T$. Thus $(s_0\sK_S)^{a,0}(X)\cong H^a(X_\Zar,\Z)$ for $X\in\Sm/S$ if $S$ is smooth over a perfect field. Thus, for $S$ smooth over a perfect field, the map $\sC\sK_S^{a,1}(X)\to \sK_S^{a,1}(X)=K_{2-a}(X)$ is an isomorphism for $a<2$, an injection for $a=2$ and $\sC\sK_S^{a,1}(X)=0$ for $a>2$. 
\end{rem}

\begin{rem}\label{rem:Orient} For $n\ge1$, let $\P^n_*$ denote the $S$-scheme $\P^n_S$, pointed by $(1:1:0\ldots:0)$ and let $\P^\infty$ denote the  colimit (in $\Spc_\bullet(S)$) of the $\P^n_*$ under the linear embeddings $(x_0:\ldots:x_n)\mapsto (x_0:\ldots:x_n:0)$. Recall from \cite[definition 1.2]{PaninPimenovRoendigs} that an {\em orientation} on a commutative ring spectrum $\sE$ in $\SH(S)$ is given by an element $c\in \sE^{2,1}(\P^\infty)$ such that the restriction $c_{|\P^1}\in \sE^{2,1}(\P^1_*)=\sE^{2,1}(T)$ is $\Sigma_T(1)$ (we use the opposite sign convention from {\it loc. cit.}).

Consider the  sequence of elements $1-[\sO_{\P^n}(-1)]\in  \sK_S^{2,1}(\P_*^n)=\tilde{K}_0(\P_S^n)$, which are clearly compatible with respect to restriction via the hyperplane embeddings $\P^n\to \P^{n+1}$. As the sequence $n\mapsto \sK_S^{1,1}(\P_S^n)=K_1(\P_S^n)$ satisfies the Mittag-Leffler condition (all the restriction maps are surjective), this defines  a unique element  $c_\sK\in \sK_S^{2,1}(\P^\infty)$, giving the standard orientation for  $\sK_S$; the fact that $c_\sK$ is an orientation follows from remark~\ref{rem:Conventions}.  
\end{rem}

\begin{rem}\label{rem:OrientUniv} The algebraic cobordism spectrum $\MGL$ has been studied in \cite{PaninPimenovRoendigs}. $\MGL$ is the $T$-spectrum $(\MGL_0, \MGL_1,\ldots)$ with $\MGL_n$ the Thom space $Th(E_n)$, with $E_n\to B\GL_n$ the universal $n$-plane bundle. $\MGL_S$ is a commutative ring spectrum object (i.e. a commutative unital monoid) in $\SH(S)$ with an orientation $c_\MGL\in\MGL^{2,1}(\P^\infty)$ given  by the diagram
\[
\xymatrix{E_1\ar[d]\ar[r]&Th(E_1)=\MGL_1\\
\P^\infty}
\]
noting that $E_1\to B\GL_1=\P^\infty$ is an isomorphism in $\sH(S)$. The main result, theorem 1.1, of \cite{PaninPimenovRoendigs} is the universality of $(\MGL, c_\MGL)$ in case $S=\Spec k$: For $\sE$ a commutative unital monoid in $\SH(k)$, sending a unital monoid morphism $\phi:\MGL\to \sE$ to $\phi(c_\MGL)$ gives a bijection of the set of unital monoid maps $\phi$ with the set of orientations $c_\sE\in \sE^{2,1}(\P^\infty)$.
\end{rem}

\begin{lem} \label{lem:orient} Suppose that $S$ satisfies the hypothesis of lemma~\ref{lem:0slice}. Suppose in addition that the map $\sK_S^{0,0}(\P^n_S)\to (s_0\sK_S)^{0,0}(\P^n_S)$ is the rank homomorphism $K_0(\P^n_S)\to H^0(S,\Z)$ (up to sign).  Then there is a unique element $c_{\sC\sK}\in \sC\sK^{2,1}(\P^\infty)$ mapping to $c_\sK$ under the canonical map $\sC\sK^{2,1}(\P^\infty)\to \sK^{2,1}(\P^\infty)$. Furthermore, $c_{\sC\sK}\in \sC\sK^{2,1}(\P^\infty)$ defines an orientation for $\sC\sK$.
\end{lem}

\begin{proof} By our assumptions on $S$, we have  the exact sequence   
\[
0\to \sC\sK_S^{2,1}(\P_S^n)\to K_0(\P_S^n)\xrightarrow{rnk}H^0(S,\Z)\to0
\]
Thus $\sC\sK_S^{2,1}(\P_*^n)\to \sK_S^{2,1}(\P_*^n)$ is an isomorphism.   Furthermore, the canonical map $\sC\sK_S^{1,1}(\P_S^n)\to \sK_S^{1,1}(\P^n)=K_1(\P_S^n)$ is an isomorphism. By the projective bundle formula,  the projective system of groups $n\mapsto K_1(\P^n_S)$ satisfies the Mittag-Leffler condition. Thus, passing to the limit over $n$ gives us the  isomorphism
\[
\sC\sK_S^{2,1}(\P_S^\infty)\cong \sK_S^{2,1}(\P_S^\infty)
\]
and the orientation $c_\sK=(1-[\sO_{\P_S^n}(-1)])_n\in  \sK_S^{2,1}(\P^\infty)$  gives us the element $c_{\sC\sK}\in \sC\sK^{2,1}(\P^\infty)$. The fact that $\Sigma_{\P^1}(1_{\sK_S})=c_{\sK|\P^1}$ implies $\Sigma_{\P^1}(1_{\sC\sK_S})=c_{\sC\sK|\P^1}$, hence $c_{\sC\sK}$ is an orientation.
\end{proof}

\begin{rem}\label{rem:Smooth} By remark~\ref{rem:SmoothCompatibiilty}, if $S$ is smooth over a perfect field, then the hypotheses of lemma~\ref{lem:orient}  are fulfilled and hence $\sC\sK_S$ has  a unique orientation mapping to the standard orientation of $\sK_S$.
\end{rem}

\section{An explicit model} \label{sec:ExplicitModel} We take $S=\Spec k$, $k$ a perfect field. Let $\Delta^*$ be the cosimplicial scheme $n\mapsto \Delta^n_k$, with
\[
\Delta^n_k:=\Spec k[t_0,\ldots, t_n]/\sum_it_i-1.
\]
The morphism $\Delta(g):\Delta^n\to \Delta^m$ associated to $g:[n]\to [m]$ is given by
\[
g^*(t_i)=\sum_{j\in g^{-1}(i)}t_j
\]
where as usual the sum over the empty index set is 0. A {\em face} of $\Delta^m$ is a closed subset $F$ defined by equations of the form $t_{i_1}=\ldots t_{i_r}=0$.

We briefly recall the construction of the homotopy coniveau tower associated to a presheaf of $S^1$-spectra on $\Sm/k$.  For $X\in \Sm/k$ and $n,m \ge0$ integers, let $\sS_X^n(m)$ denote the set of closed subsets $W$ of $X\times \Delta^m$ such that
\[
\codim_{X\times F}W\cap X\times F\ge n
\]
for all faces $F$ of $\Delta^m$. For $E\in \Spt_{S^1}(k)$, we let 
\[
E^{(n)}(X,m):=\hocolim_{W\in\sS_X^n(m)}E^W(X\times\Delta^m).
\]
This gives us the simplicial spectrum $m\mapsto E^{(n)}(X,m)$, and the associated total spectrum $E^{(n)}(X)$. This construction is contraviarantly  functorial in $X$ for equi-dimensional morphisms. Letting $\Sm//k\subset \Sm/k$ denote the subcategory of $\Sm/k$ with the same objects, and with morphisms the smooth morphisms, sending $X$ to $E^{(n)}(X)$ defines a presheaf of spectra on $\Sm//k$. . It was shown in \cite[theorem 4.1.1]{LevineHC} that there are models $\tilde{E}^{(n)}(X)$ for $E^{(n)}(X)$ so that $X\to \tilde{E}^{(n)}(X)$ extends to a presheaf of spectra $\tilde{E}^{(n)}$ on $\Sm/k$, isomorphic to $X\mapsto E^{(n)}(X)$ on $\Sm//k$. The main result of \cite{LevineHC} is:

\begin{thm}[\hbox{\cite[theorem 7.1.1]{LevineHC}}] \label{thm:HC} Let $E\in \Spt_{S^1}(k)$ be quasi-fibrant. There is a natural isomorphism in $\SH_{S^1}(k)$
\[
f_nE\cong \tilde{E}^{(n)}.
\]
\end{thm}
Here {\em quasi-fibrant} means that a fibrant replacement $E\to E^{fib}$ gives a weak equivalence $E(X)\to E^{fib}(X)$ for all $X\in\Sm/k$.

For a category $\sC$, we let $\sC*$ be $\sC$ with a final object $*$ adjoined. Let $I$ be a finite category and $\sX:I\to \Sm//k*$ an $I$-diagram in $\Sm//k$, that is, a functor. For  a presheaf of spectra $E$ on $\Sm//k$,   define $E(\sX)$ as
\[
E(\sX):=\holim_{I^\op} E\circ \sX^\op,
\]
where $E(*)$ is defined to be the initial object in $\Spt$. For an $I$-diagram $\sX:I\to \Sm//k*$, we have the $I$-diagram  $\Sigma_T^\infty \sX_+:I\to \Spt_T(k)$ defined by $\Sigma_T^\infty \sX_+(i):=\Sigma_T^\infty \sX(i)_+$ if $\sX(i)$ is in $\Sm/k$, and setting $\Sigma_T^\infty *_+$ equal to the final $T$-spectrum $(pt, pt,\ldots)$. We similarly define $\Sigma_s^\infty \sX_+:I\to \Spt_{S^1}(k)$.

\begin{ex}\label{ex:Supports} Let $X$ be in $\Sm/k$ and $j:U\to X$ an open immersion with closed complement $Z$. Let $I$ be the category 
\[
\xymatrix{
0\ar[r]^a\ar[d]&1\\
{*}}
\]
and let  $X/U:I\to \Sm//k$ be the diagram $0\mapsto U$, $1\mapsto X$, $*\mapsto *$, $a\mapsto j$. Then for $E$ a presheaf of spectra on $\Sm//k$, $E(X/U)$ is just the homotopy fiber of $j^*:E(X)\to E(U)$. Similarly, if $I$ is the one-point category $0$ and $\sX$ is the functor $\sX:0\to \Sm//k$ with $\sX(0)=X$, we have a canonical isomorphism $E(X)\cong E(\sX)$ in $\SH$. 
\end{ex}

\begin{lem} \label{lem:Rep} Let $\sX:I\to \Sm//k*$ be a finite diagram of smooth $k$-schemes (possibly with $\sX(i)=*$ for some values $i\in I$) and take $\sE\in \SH(k)$. Let $E\in \Spt_{S^1}(k)$ be a fibrant model for $\Omega^\infty_T\sE\in \SH_{S^1}(k)$. Then 
$\sHom(\hocolim_I\Sigma_T^\infty \sX_+, f_n\sE)\in \SH$ is represented by the spectrum $E^{(n)}(\sX)$.  
\end{lem}

\begin{proof}  The adjunction 
\[
\Sigma^\infty_T:\SH_{S^1}(k)\leftrightarrow\SH(k):\Omega^\infty_T
\]
gives the isomorphism in  $\SH$
\[
\sHom(\hocolim_I\Sigma_T^\infty \sX_+, f_n\sE)\cong \sHom(\hocolim_I\Sigma_s^\infty \sX_+, \Omega^\infty_T f_n\sE).
\]
It follows from  \cite[theorem 7.1.1, theorem 9.0.3]{LevineHC}  that we have the isomorphism in $\SH_{S^1}(k)$
\[
\Omega^\infty_T f_n\sE\cong f_n\Omega^\infty_T\sE=f_n E.
\]
We thus  have the isomorphisms in $\SH$
\begin{align*}
 \sHom(\hocolim_I\Sigma_s^\infty \sX_+, \Omega^\infty_T f_n\sE)&\cong 
  \sHom(\hocolim_I\Sigma_s^\infty \sX_+,   f_nE)\\
  &\cong \holim_I\sHom(\Sigma_s^\infty \sX_+,   f_nE)\\
  & \cong \holim_I f_nE\circ \sX^\op.
  \end{align*}
We give the category of $I$-diagrams in $\Spt$ the projective model structure, with weak equivalences the pointwise ones, and let $Ho(\Spt_I)$ denote the homotopy category. By \cite[theorem 7.1.1]{LevineHC}, we have the isomorphism  in $Ho(\Spt_I)$ 
\[
f_nE\circ \sX^\op\cong E^{(n)}\circ \sX^\op,
\]
giving the isomorphism in $\SH$
\[
\sHom(\hocolim_I\Sigma_s^\infty \sX_+, \Omega^\infty_T f_n\sE)\cong \holim_I E^{(n)}\circ \sX^\op=E^{(n)}(\sX).
\]
\end{proof}

We let $K\in \Spt_{S^1}(k)$ be the presheaf of spectra given by sending $X$ to the Quillen-Waldhausen spectrum $K(X)$ representing the algebraic $K$-theory of $X$. We note that $K$ is a quasi-fibrant object of $\Spt_{S^1}(k)$.

\begin{prop}\label{prop:Comp1} Let $\sX:I\to \Sm//k*$ be a finite diagram of smooth $k$-schemes  as in lemma~\ref{lem:Rep}. Take $X$ in $\Sm/k$. There is a canonical  isomorphism
\[
\sC\sK^{p,q}(\sX)\cong \pi_{2q-p}K^{(q)}(\sX), p, q\in \Z,
\]
natural in $\sX$. 
\end{prop}

\begin{proof} We make $X\mapsto \pi_{2q-p}K^{(q)}(X)$ a functor in $X$ by using the functorial model $\tilde{K}^{(q)}$ for $K^{(q)}$ and the canonical isomorphism $ \pi_*K^{(q)}(X)\cong  \pi_*\tilde{K}^{(q)}(X)$.

Using lemma~\ref{lem:Rep} and a variety of adjunctions and definitions, we have the sequence of isomorphisms
\begin{align*}
\sC\sK^{p,q}(\sX)&=\Hom_{\SH(k)}(\hocolim_I\Sigma_T^\infty \sX_+,\Sigma^{p,q}\sC\sK)\\
&=\Hom_{\SH(k)}(\hocolim_I\Sigma_T^\infty \sX_+,\Sigma^{p,q}f_0\sK)\\
&=\Hom_{\SH(k)}(\hocolim_I\Sigma_T^\infty \sX_+,\Sigma_{S^1}^{p-2q}\Sigma_T^qf_0\sK)\\
&=\Hom_{\SH(k)}(\hocolim_I\Sigma_T^\infty \sX_+,\Sigma_{S^1}^{p-2q}f_q\Sigma^q_T\sK)\\
&=\Hom_{\SH(k)}(\Sigma^\infty_T\Sigma_{S^1}^{2q-p}\hocolim_I\Sigma_s^\infty \sX_+,f_q\sK)\\
&=\Hom_{\SH_{S^1}(k)}(\Sigma_{S^1}^{2q-p}\hocolim_I\Sigma_s^\infty \sX,\Omega^\infty_Tf_q\sK)\\
&=\Hom_{\SH_{S^1}(k)}(\Sigma_{S^1}^{2q-p}\hocolim_I\Sigma_s^\infty \sX,f_qK)\\
&=\Hom_{\SH}(S^{2q-p},\sHom(\hocolim_I\Sigma_T^\infty \sX_+, f_qK))\\
&=\pi_{2q-p}K^{(q)}(\sX).
\end{align*}
\end{proof}

\begin{cor} \label{cor:CKComp} Let $X$ be smooth over $k$. Then $\sC\sK^{2n,n}(X)$ is equal to $K_0(X)$ for $n\le 0$. For $n>0$, 
$\sC\sK^{2n,n}(X)$ is determined by the exact sequence
\[
K_0^{(n)}(X,1)\xrightarrow{\delta_1^*-\delta_0^*}K_0(\sM_X^{(n)})\to \sC\sK^{2n,n}(X)\to 0.
\]
In particular, 
\[
\sC\sK^{2n,n}(k)\to \sK^{2n,n}(k)=K_0(k)=\Z
\]
is an isomorphism for $n\le 0$;  for $n>0$, $\sC\sK^{2n,n}(k)=0$. 
\end{cor}

\begin{proof} Viewing $X$ as the one-point diagram $*\mapsto X$, it follows from proposition~\ref{prop:Comp1} that we have a canonical isomorphism $\sC\sK^{2n,n}(X)\cong \pi_0(K^{(n)}(X))$. For $n\le 0$, $K^{(n)}(X)$ is weakly equivalent to $K(X)$. For $n>0$, $K^{(n)}(X,m)$ is -1 connected, since $K^W(Y)$ is so for all closed $W\subset Y$, $Y\in\Sm/k$. This gives the presentation for $\sC\sK^{2n,n}(X)=\pi_0(K^{(n)}(X))$. 
\end{proof}

For $X\in\Sm/k$ with generic point $\eta$ we define $\sC\sK^{a,b}(\eta)$ as the stalk at $\eta$ of the presheaf on $X_\Zar$, $U\mapsto \sC\sK^{a,b}(U)$. 

\begin{cor}\label{cor:CKComp2} Let $\eta$ be a generic point of some $X\in\Sch/k$. Then $\sC\sK^{p,q}(\eta)=0$ for $p>q$.
\end{cor}

\begin{proof} By proposition~\ref{prop:Comp1}, $\sC\sK^{p,q}(\eta)\cong \pi_{2q-p}K^{(q)}(\eta)$. But $K^{(q)}(\eta)$ is the total spectrum of the simplicial spectrum $m\mapsto K^{(q)}(\eta,m)$, with 
\[
K^{(q)}(\eta,m):=\hocolim_{W\in\sS_X^n(m)}K^W(\eta\times\Delta^m).
\]
As $\eta\times\Delta^m\cong  \Delta^m_{k(\eta)}$, it follows that  each closed subset $W$ of $\eta\times\Delta^m$ has $\codim W\le m$, and hence $K^{(q)}(\eta,m)$ is the 0-spectrum if $m<q$. Furthermore, the $K$-theory presheaf is a presheaf of -1 connected spectra and for each open immersion $j:U\to V$ in $\Sm/k$, the restriction map $j^*K_0(V)\to K_0(U)$ is surjective. Thus $K^W(\eta\times\Delta^m)$ is a -1 connected spectrum for each $m$. Using the strongly convergent spectral sequence
\[
E^1_{a,b}=\pi_bK^{(q)}(\eta,a)\Longrightarrow \pi_{a+b}K^{(q)}(\eta),
\]
we see that $\pi_nK^{(q)}(\eta)=0$ for $n<q$, hence 
\[
\sC\sK^{p,q}(\eta)\cong \pi_{2q-p}K^{(q)}(\eta)=0
\]
for $p>q$.
\end{proof}

We conclude this section with a discussion of the functor $\sC\sK^{2q-1, q}$. Let $\sE\in\SH(S)$ represent a bi-graded cohomology theory. Let $t_\sE\in \sE^{1,1}(\G_m)$ be the element corresponding to the unit $1\in \sE^{0,0}(S)$ under the suspension isomorphism. By functoriality, $t_\sE$ gives a map of pointed sets
\[
t_\sE(X):\sO_X^\times(X)\to \sE^{1,1}(X);
\]
if $\sE$ admits an orientation $c_\sE\in \sE^{2,1}(\P^\infty)$ (which we will from now on assume), then $t_\sE(X)$ is a group homomorphism.\footnote{Letting $\mS$ denote the sphere spectrum and writing $[a]:=t_\mS(a)$, this follows from the identity $[ab]=[a]+[b]+H[a][b]$ ($H:\mS\wedge\G_m\to \mS$ the stable Hopf map) and the fact that $H$ goes to zero in any oriented theory $\sE$. Both these facts are proven by Morel in \cite[\S 6]{MorelLec}.} Using the $\sE^{*,*}(S)$-module structure on $\sE^{*,*}(X)$, $t_\sE(X)$
extends to a map of $\sE^{*,*}(S)$-modules
\[
t_\sE(X):\sE^{2*,*}(S)\otimes_\Z\sO_X^\times(X)\to \sE^{2*+1,*+1}(X).
\]
 
\begin{lem} \label{lem:unit} Suppose $S=\Spec k$. Let $\eta$ be a generic point of some $X\in\Sm/k$. \\\\
1. Take $\sE=\sK$. Then  $\sK^{2*,*}(k)\cong \Z[\beta,\beta^{-1}]$, $\deg\beta=-1$ and $t_\sK(\eta):\sK^{2*,*}(k)\otimes_\Z k(\eta)^\times\to \sK^{2*+1,*+1}(\eta)$ is an isomorphism.\\\\
2. Take $\sE=\sC\sK$. Then  $\sC\sK^{2*,*}(k)\cong \Z[\beta]$, $\deg\beta=-1$ and $t_{\sC\sK}(\eta):\sC\sK^{2*,*}(k)\otimes_\Z k(\eta)^\times\to \sC\sK^{2*+1,*+1}(\eta)$ is an isomorphism.\\\\
3. Take $\sE=\MGL$ and suppose $k$ has characteristic zero. Then $\MGL^{2*,*}(k)$ is canonically isomorphic to the Lazard ring $\L^*$, $t_\MGL(\eta):k(\eta)^\times\to \MGL^{1,1}(\eta)$ is an isomorphism and $t_\MGL(\eta):\MGL^{2*,*}(k)\otimes_\Z k(\eta)^\times\to \MGL^{2*+1,*+1}(\eta)$ is surjective.
\end{lem}

\begin{proof} We have already seen in remark~\ref{rem:Conventions} that under the isomorphism $\sK^{1,1}(\G_m)\cong K_1(\G_m)$, $t_\sK$ goes to the class of the canonical unit $t$ (or possibly $t^{-1}$).  By functoriality, the map $t_\sK:k(\eta)^\times\to K_1(\eta)=\sK^{1,1}(\eta)$ is the usual isomorphism $k(\eta)^\times\cong K_1(\eta)$ sending $x\in k(\eta)^\times$ to the class of the automorphism $\times x^\epsilon:k(\eta)\to k(\eta)$, where $\epsilon=\pm 1$ is universal choice of sign (independent of $k$ or $\eta$).

The isomorphism $\sK^{2*,*}(k)\cong \Z[\beta,\beta^{-1}]$ follows from the Bott periodicity isomorphism $\sK^{2n,n}(k)\cong K_0(k)\cong \Z$. Since $\beta$ is invertible, the fact that $t_\sK: k(\eta)^\times\to \sK^{1,1}(\eta)$ is an isomorphism implies that $t_\sK(\eta):\sK^{2*,*}(k)\otimes_\Z k(\eta)^\times\to \sK^{2*+1,*+1}(\eta)$ is an isomorphism.

For (2), it follows from the universal  property of $f_0\to\id$ that $\sC\sK^{a,b}(X)\to \sK^{a,b}(X)$ is an isomorphism for all $b\le 0$, $a\in \Z$, $X\in\Sm/k$. In particular,  $\sC\sK^{2b-1,b}(X)=\sK^{2b-1,b}(X)=K_1(X)$ for $b\le0$. For $b>1$, $\sC\sK^{2b-1,b}(\eta)=0$ by corollary~\ref{cor:CKComp2}, and $\sC\sK^{1,1}(X)\to \sK^{1,1}(X)$ is an isomorphism by lemma~\ref{lem:0slice}.  Similarly, the map $\sC\sK^{2n,n}(X)\to \sK^{2n,n}(X)=\Z\beta^{-n}$ is an isomorphism for $n\le 0$ and by corollary~\ref{cor:CKComp2}, $\sC\sK^{2n,n}(\eta)=0$ for $n>0$. Thus the map $\sC\sK^{2*,*}(k)\to \sK^{2*,*}(k)$ identifies  $\sC\sK^{2*,*}(k)$ with the subring $\Z[\beta]$ of $\sK^{2*,*}(k)=\Z[\beta,\beta^{-1}]$. Putting this all together, (1) implies (2).

For (3), the orientation $c_\sK$ gives the canonical morphism of oriented ring $T$-spectra $\vartheta_\sK:\MGL\to \sK$ \cite[theorem1.1]{PaninPimenovRoendigs}, inducing the commutative diagram
\[
\xymatrix{
\sO_X(X)^\times\ar[r]^{t_\MGL}\ar[dr]_{t_\sK}&\MGL^{1,1}(X)\ar[d]^{\vartheta_\sK}\\
&\sK^{1,1}(X).}
\]
As $t_\sK:k(\eta)^\times\to \sK^{1,1}(\eta)=K_1(\eta)$ is an isomorphism, it follows that $t_\MGL:k(\eta)^\times\to \MGL^{1,1}(\eta)$ is injective. The isomorphism $\L^*\to \MGL^{2*,*}(k)$ and the surjectivity of 
 $t_\MGL:\L^*\otimes k(\eta)^\times\to \MGL^{2*+1,*+1}(k(\eta))$ follow from the Hopkins-Morel spectral sequence \cite{HopkinsMorel, Hoyois}.
\end{proof}

\section{Oriented duality theories}\label{sec:Orient}

Recall from \cite[\S 1]{LevineOrient} the category $\SP/k$ of {\em smooth pairs} over $k$, with objects $(M,X)$, $M\in \Sm/k$ and $X\subset M$ a closed subset; a morphism $f:(M,X)\to (N,Y)$ is a morphism $f:M\to N$ in $\Sm/k$ such that $f^{-1}(Y)\subset X$. We let $\Sch/k$ denote the category of quasi-projective $k$-schemes;  for a full subcategory $\sV$ of $\Sch/k$, let $\sV'$  be the subcategory of $\sV$ with the same objects and morphisms the projective morphisms.

Building on work of  Mocanasu \cite{Mocanasu} and Panin \cite{Panin2}, we have defined in \cite[definition 3.1]{LevineOrient} the notion of a bi-graded {\em  oriented duality theory} $(H, A)$ on $\Sch/k$. Here $A$ is a bi-graded oriented cohomology theory on $\SP/k$, $(M,X)\mapsto A_X^{*,*}(M)$, and $H$ is a functor from $\Sch'/k$ to bi-graded abelian groups.  The oriented cohomology theory $A$ satisfies the axioms listed in \cite[definitions 1.2,  1.5]{LevineOrient}. In particular, $(M,X)\mapsto A_X^{*,*}(M)$ admits a long exact sequence 
\[
\ldots\to A_X^{*,*}(M)\to A^{*,*}(M)\to A^{*,*}(M\setminus X)\xrightarrow{\del}A_X^{*+1,*}(M)\to\ldots
\]
where for instance $A^{*,*}(M):=A^{*,*}_M(M)$ and the boundary map $\del$ is part of the data. In addition, there is an excision property and a homotopy invariance property. The ring structure is given by external products and pull-back by the diagonal. The orientation is given by a collection of isomorphisms $\Th^E_X:A_X(M)\to A_X(E)$, for $(M,X)\in \SP/k$ and $E\to M$ a vector bundle, satisfying the axioms of \cite[def. 3.1.1]{Panin2}. We extend some of the results of \cite{Panin2} in \cite[theorem 1.12, corollary 1.13]{LevineOrient} to show that the data of an orientation is equivalent to giving well-behaved push-forward maps $f_*:A_X(M)\to A_Y(N)$ for $(M,X), (N,Y)\in\SP/k$, with the meaning of ``well-behaved" detailed in \cite[\S1]{LevineOrient}.

The homology theory $H$ comes with restriction maps $j^*:H_{*,*}(X)\to H_{*,*}(U)$ for each open immersion $j:U\to X$ in $\Sch/k$, external products $\times:H_{*,*}(X)\otimes H_{*,*}(Y)\to H_{*,*}(X\times Y)$,
 boundary maps $\del_{X,Y}:H_{*,*}(X\setminus Y)\to H_{*-1,*}(Y)$ for each closed subset  $Y\subset X$, , isomorphisms $\alpha_{M,X}:H_{*,*}(X)\to A^{2m-*, m-*}_X(M)$ for each $(M,X)\in \SP/k$, $m=\dim_kM$, and finally cap product maps
\[
f^*(-)\cap:A_X^{a,b}(M)\otimes H_{*,*}(Y)\to H_{*-a, *-b}(Y\cap f^{-1}(X)
\]
for $(M,X)\in \SP/k$, $f:Y\to X$ a morphism in $\Sch/k$. These satisfy a number of axioms and compatibilities (see \cite[\S3]{LevineOrient} for details), which essentially say that a structure for $A^{*,*}_X(M)$ is compatible with the corresponding structure for $H_{*,*}(X)$ via the isomorphism $\alpha_{M,X}$. Roughly speaking, this is saying that a particular structure for  $A^{*,*}_X(M)$ depends only on $X$ and not the choice of embedding $X\hookrightarrow M$.

\begin{rem}\label{rem:FGL} Let $L\to Y$ be a line bundle on some $Y\in\Sm/k$ with 0-section $0:Y\to L$. For an  oriented cohomology theory $A$ one has the element
\[
c_1^A(L):=0^*(0_*(1^A_Y)),
\]
where $1^A_Y\in A^0(Y)$ is the unit element.  As pointed out in \cite[corollary 3.3.8]{Panin2}, or as noted in \cite[remark 1.17]{LevineOrient}, for line bundles $L$, $M$ on some $Y\in\Sm/k$,  the elements $c_1(L), c_1(M)\in A^1(Y)$ are nilpotent, and commute with one another, hence for each power series $F(u,v)\in A^*(k)[[u,v]]$  the evaluation $F(c_1(L), c_1(M))$ gives a well-defined element of $A^*(Y)$. In addition, the cohomology theory $A$ has a unique associated formal group law $F_A(u,v)\in A^*(k)[[u,v]]$ with
\[
F_A(c_1(L), c_1(M))=c_1(L\otimes M)
\]
for all line bundles $L$, $M$ on $Y\in\Sm/k$.

For an $X\in\Sch/k$ with line bundle $L\to X$, the quasi-projectivity of $X$ implies that  $X$ admits a closed immersion $i:X\to M$ for some $M\in \Sm/k$ such that $L$ extends to a line bundle $\sL\to M$. One can then define 
\[
\tilde{c}_1(L):H_*(X)\to H_{*-1}(X)
\]
via the product
\[
(-)\cdot c_1(\sL):A_X^*(M)\to A_X^{*+1}(M)
\]
and the isomorphisms $H_*(X)\cong A^{d_M-*}_X(M)$. One shows that this is independent of the choice of $(M,\sL)$, giving the well-defined operator $\tilde{c}_1(L)$. 
\end{rem} 

The main example of  oriented duality theory $(H, A)$ is given by an oriented $T$-ring spectrum  $\sE$ in $\SH(k)$, assuming $k$ is a field admitting resolution of singularities (e.g., characteristic zero), defined by taking
\[
\sE^{a,b}_X(M):=\Hom_{\SH(k)}(\Sigma^\infty_T(M/M\setminus X), \Sigma^{a,b}\sE),
\]
i.e., the usual bi-graded cohomology with supports. For each $X\in\Sch/k$, choose a closed immersion of $X$ into a smooth $M$ and set $\sE'_{a,b}(X):=\sE^{2m-a,m-b}_X(M)$, where $m=\dim_kM$. The fact that $(M,X)\mapsto \sE^{*,*}_X(M)$ defines an oriented bi-graded ring cohomology theory is proved just as in the case of $\sE=\MGL$, which was discussed in \cite[\S 4]{LevineOrient}; the main point is Panin's theorem \cite[theorem 3.7.4]{Panin2}, which says that an orientation for $\sE$ (in the sense of  remark~\ref{rem:Orient}) defines an orientation in the sense of ring cohomology theories for the bi-graded $\sE$-cohomology with supports. 

The fact that the formula given above for the homology theory $\sE'_{*,*}$ is well-defined and extends to make $(\sE'_{*,*}(-), \sE^{*,*}_{-}(-))$ a bi-graded  oriented duality theory is \cite[theorem 3.4]{LevineOrient}. The essential point is to show that the cohomology with support $\sE^{2d-*, d-*}_X(M)$, for $X\hookrightarrow M$ a closed immersion of some $X$ in a smooth $M$ of dimension $d$, depends (up to canonical isomorphism) only on $X$, and similarly, given a projective morphism $f:Y\to X$ in $\Sch/k$, there are smooth pairs $(M,X)$, $(N,Y)$, an extension of $F$ to a morphism $F:N\to M$ and the map $F_*:\sE_Y^{2d_N-*, d_N-*}(N)\to \sE_X^{2d_M-*, d_M-*}(M)$ is independent  (via the canonical isomorphisms $\sE'_{*,*}(Y)\cong \sE_Y^{2d_N-*, d_N-*}(N)$, 
$\sE'_{*,*}(X)\cong \sE_X^{2d_M-*, d_M-*}(M)$) of the choices. The other structures for  $\sE'_{*,*}(-)$ are defined similarly via the $\sE$-cohomology with supports, and one has the corresponding independence of any choices.

It follows directly from the construction of $\sE'$ that the assignment $(\sE, c_\sE)\mapsto (\sE', \sE)$ is  functorial in the oriented cohomology theory  $(\sE, c_\sE)$.  In particular, let $\ch:\MGL\to\sE$ be a morphism of oriented cohomology theories, that is, $\ch$ is a morphism in $\SH(k)$, compatible with the ring-object structures of $\MGL$ and $\sE$, and compatible with 1st Chern  classes. Then we have an extension of $\ch$ to a natural transformation of oriented duality theories
\[
(\ch',\ch):(\MGL',\MGL)\to (\sE',\sE)
\]

\begin{rem}\label{rem:KThyFGL} As shown in lemma~\ref{lem:unit} , the coefficient rings for $\sK$ and $\sC\sK$ are $\sK^{2*,*}(k)=\Z[\beta, \beta^{-1}]$ and $\sC\sK^{2*,*}(k)=\Z[\beta]$, respectively, with $\beta$ having degree $-1$. For  $\sK$,  the orientation $c_\sK$ restricted to $\P^n$ is given by the class of $1-[\sO(-1)]\in K_0(\P^n)\cong \sK^{2,1}(\P^n)$. It follows (by functoriality and Jouanoulou's trick) that for a line bundle $L$ on some $X\in\Sm/k$, the 1st Chern class is given by $c_1^\sK(L)=\beta^{-1}(1-[L^{-1}])$ (where we consider $1, [L^{-1}]\in \sK^{0,0}(X)=K_0(X)$). A direct calculation gives the formal group law for $(\sK'_{2*,*},\sK^{2*,*})$ as
 $(F_\sK(u,v)=u+v-\beta\cdot uv, \Z[\beta,\beta^{-1}])$. Since the orientation for $\sK$ lifts to that of $\sC\sK$, it follows that the formal group law for  $(\sC\sK'_{2*,*}, \sC\sK^{2*,*})$ is $(u+v-\beta\cdot uv, \Z[\beta])$.
\end{rem}

\section{Algebraic cobordism and oriented duality theories}\label{sec:AlgCobordOrient}
We recall the theory of {\em algebraic cobordism} $X\mapsto \Omega_*(X)$, $X\in \Sch/k$. For each $X\in \Sch/k$, $\Omega_n(X)$ is an abelian group with generators $(f:Y\to X)$, $Y\in \Sm/k$ irreducible of dimension $n$ over $k$ and $f:Y\to X$ a projective morphism.  $\Omega_*$ is the universal {\em oriented Borel-Moore homology theory} on $\Sch/k$; this consists of the data of a functor from $\Sch/k'$ to graded abelian groups, external products, first Chern class operators $\tilde{c}_1(L):\Omega_*(X)\to \Omega_{*-1}(X)$ for $L\to X$ a line bundle, and pull-back maps $g^*:\Omega_*(X)\to \Omega_{*+d}(Y)$ for each \lci morphism $g:Y\to X$ of relative dimension $d$. These of course satisfy a number of compatibilities and additional axioms. 

For an oriented duality theory $(H,A)$ on $\Sch/k$ and $Y$ in $\Sm/k$ of dimension $d$ over $k$, the {\em fundamental class} $[Y]_{H,A}\in H_d(Y)$ is the image of the unit $1_Y\in A^0(Y)$ under the inverse of the isomorphism $\alpha_Y:H_d(Y)\to A^0(Y)$. For an oriented Borel-Moore homology theory $B$ on $\Sch/k$, we similarly have the fundamental class $[Y]_B\in B_d(Y)$ defined by $[Y]_B:=p^*(1)$, where $1\in B_0(\Spec k)$ is the unit and $p:Y\to \Spec k$ the structure morphism.

We recall the following result from \cite{LevineOrient}:

\begin{prop}[\hbox{\cite[propositions 4.2, 4.4, 4.5]{LevineOrient}}] \label{prop:Universal} Let  $k$ be a field admitting resolution of singularities and let $(H,A)$ be a $\Z$-graded oriented duality theory on $\Sch/k$. \\\\
1. There is a unique natural transformation $\vartheta_H:\Omega_*\to H_*$ of functors $\Sch/k'\to \GrAb$, such that $\vartheta_H(Y)$ is compatible with fundamental classes for $Y\in\Sm/k$. In addition, $\vartheta_H$ is compatible with pull-back maps for open immersions in $\Sch/k$,  with 1st Chern class operators,  with external products and with cap products. \\\\
2. For $Y\in\Sm/k$, the map $\vartheta^A(Y):\Omega^*(Y)\to A^*(Y)$ induced by $\vartheta_H$, the identity $\Omega^*(Y)=\Omega_{\dim Y-*}(Y)$ and the isomorphism $\alpha_Y:H_{\dim Y-*}(Y)\to A^*(Y)$  is a ring homomorphism and is compatible with pull-back maps for arbitrary morphisms in $\Sm/k$. Finally, one has
\[
\vartheta^A(Y)(c_1^\Omega(L))=c_1^A(L)
\]
for each line bundle $L\to Y$.
\end{prop}

\begin{rem}\label{rem:CompClassFGL} We have already noted that one has a formal group law $F_A(u,v)\in A^*(k)[[u,v]]$ associated to the oriented cohomology theory $A$. Similarly, for each oriented Borel-Moore homology theory $B$ on $\Sch/k$, there is an associated  formal group law $F_B(u,v)\in B_*(k)[[u,v]]$, characterised by the identity $F_B(c_1(L), c_1(M))=c_1(L\otimes M)$ for each pair of line bundles $L,M$ on some $Y\in Sm/k$ (this follows from \cite[corollary 4.1.8, proposition 5.2.1, proposition 5.2.6]{LevineMorel}).  Letting $\phi_A:\L^*\to A^*(k)$, $\phi_B:\L^*\to B^*(k)$ denote the classifying maps associated to $F_A$, $F_B$, respectively, suppose that $A$ extends to an oriented duality theory $(H,A)$. Then
\begin{equation}\label{eqn:FGLTransform}
\vartheta^A(F_\Omega)=F_A. 
\end{equation}
Indeed, $F_A$ is characterised by identity $F_A(c^A_1(L), c^A_1(M))=c^A_1(L\otimes M)$ for each pair of line bundles $L,M$ on some $Y\in Sm/k$, and since $\vartheta^A(c^\Omega_1(N))=c_1^A(N)$ for each line bundle $N\to Z$, $Z\in \Sm/k$, the fact that $F_\Omega(c^\Omega_1(L), c^\Omega_1(M))=c^\Omega_1(L\otimes M)$ combined with proposition~\ref{prop:Universal}(2) yields the identity \eqref{eqn:FGLTransform}.

Via the universal property of the Lazard ring, the relation \eqref{eqn:FGLTransform} is equivalent to the identity
\begin{equation}\label{eqn:LazardTransform}
\vartheta^A\circ \phi_\Omega=\phi_A.
\end{equation}

Finally, we recall that the classifying map $\phi_\Omega:\L_*\to \Omega_*(k)$ is an isomorphism \cite[theorem 1.2.7]{LevineMorel}.
\end{rem}

\begin{cor} \label{cor:NatTrans} Let $(\sE,c_\sE)$ be pair consisting of a commutative unital monoid object $\sE\in \SH(k)$ with  orientation  $c_\sE$, and let $(\sE'_{*,*}, \sE^{*,*})$ be the corresponding bi-graded oriented duality theory. There is a  unique natural transformation 
\[
\vartheta_{(\sE,c_\sE)}:\Omega_*\to \sE'_{2*,*}
\]
of functors $\Sch/k'\to \GrAb$, such that $\vartheta_{(\sE,c_\sE)}(Y)$ is compatible with fundamental classes for $Y\in\Sm/k$. In addition, $\vartheta_{(\sE,c_\sE)}$ is compatible with pull-back maps for open immersions in $\Sch/k$, 1st Chern class operators, external products and cap products. For $Y\in\Sm/k$, the map $\vartheta^\sE(Y):\Omega^*(Y)\to \sE^{2*,*}(Y)$ induced by $\vartheta_{(\sE,c_\sE)}$ is a ring homomorphism and is compatible with pull-back maps for arbitrary morphisms in $\Sm/k$, and satisfies
\[
\vartheta_{(\sE,c_\sE)}(Y)(c_1^\Omega(L))=c_1^\sE(L)
\]
for each line bundle $L\to Y$.
\end{cor}

\begin{rem}\label{rem:Char}  By \cite[lemma 2.5.11]{LevineMorel}, $\Omega_*(X)$ is generated as an abelian group by the cobordism cycles $(f:Y\to X)$, $Y\in\Sm/k$ irreducible, $f:Y\to X$  a projective morphism.  Furthermore,  the identity $(f:Y\to X)=f_*([Y]_\Omega)$ holds in $\Omega_{\dim Y}(X)$. Thus $\vartheta_{(\sE,c_\sE)}$ is characterized by the formula
\[
\vartheta_{(\sE,c_\sE)}(f:Y\to X):=f_*^{\sE'}([Y]_{\sE',\sE}).
\]
\end{rem}

We may apply corollary~\ref{cor:NatTrans}  in the universal case: $\sE=\MGL$ with its canonical orientation. This gives us the natural transformation
\begin{equation}\label{eqn:NatMGL}
\vartheta_\MGL:\Omega_*\to \MGL'_{2*,*}.
\end{equation}

\begin{thm}[\hbox{ \cite[theorem 3.1]{LevineComparison}}] \label{thm:MGLComp} Assume that  $k$ is a field of characteristic zero.  Then the natural transformation \eqref{eqn:NatMGL} is an isomorphism.
\end{thm}

\begin{rem} This result relies on the Hopkins-Morel spectral sequence, see \cite{HopkinsMorel, Hoyois}.
\end{rem}

In the course of the proof, we proved another result which we will be using here. 

Let $X$ be in $\Sch/k$ and let $d=d_X:=\max_{X'}\dim_kX'$, as $X'$ runs over the irreducible components of $X$. We define $\MGL_{2*,*}^{\prime (1)}(X)$ by
\[
\MGL_{2*,*}^{\prime (1)}(X):=\colim_W \MGL_{2*,*}'(W)
\]
as $W$ runs over all (reduced) closed subschemes of $X$ which contain no dimension $d$ generic point of $X$; $\Omega_*^{(1)}(X)$ is defined similarly. The natural transformation $\vartheta_\MGL$ gives rise to the commutative diagram
\begin{equation}\label{eqn:HM1}
\xymatrix{
\Omega_*^{(1)}(X)\ar[d]_{\vartheta^{(1)}}\ar[r]^-{i_*}&\Omega_*(X)\ar[d]_{\vartheta(X)}\ar[r]^-{j^*}&\oplus_{\eta\in X_{(d)}}\Omega_*(k(\eta))\ar[r]\ar[d]^{\vartheta}&0\\
\MGL_{2*,*}^{\prime (1)}(X)\ar[r]_-{i_*}&\MGL_{2*,*}'(X)\ar[r]_-{j^*}&\oplus_{\eta\in X_{(d)}}\MGL_{2*,*}'(k(\eta))\ar[r]&0}
\end{equation}
with exact rows and with all vertical arrows isomorphisms. As $(\MGL', \MGL)$ is an oriented duality theory, the bottom line extends to the long exact sequence
\begin{multline*}
\ldots\to \oplus_{\eta\in X_{(d)}} \MGL_{2*+1,*}'(k(\eta))\xrightarrow{\del} \MGL_{2*,*}^{\prime (1)}(X)\\\xrightarrow{i_*}\MGL_{2*,*}'(X)\xrightarrow{j^*}\oplus_{\eta\in X_{(d)}}\MGL_{2*,*}'(k(\eta))\to0.
\end{multline*}
Furthermore, the Hopkins-Morel spectral sequence \cite{HopkinsMorel, Hoyois} 
\[
E_2^{p,q}:=\L^{-q}\otimes H^{p-q}(Y,\Z(n+q))\Longrightarrow \MGL^{p+q,n}(Y)
\]
gives a surjection for each $\eta\in X_{(d)}$
\[
t_\MGL(\eta):\L_{*-d+1}\otimes k(\eta)^\times\to  \MGL_{2*+1,*}'(k(\eta))
\]
(see lemma~\ref{lem:unit}). We have constructed in \cite[\S6]{LevineComparison} a group homomorphism
\[
\Div:\L_{*-d+1}\otimes\oplus_{\eta\in X_{(d)}} \Z[k(\eta)^\times]\to \Omega_*^{\prime(1)}(X)
\]
with $\vartheta^{(1)}\circ\Div=\del\circ\oplus_\eta t_\MGL(\eta)$. Since the maps $\vartheta^{(1)}$ and $\vartheta(X)$ are isomorphisms, the map $\Div$ factors through the surjection
\[
\L_{*-d+1}\otimes\oplus_{\eta\in X_{(d)}} \Z[k(\eta)^\times]\to \L_{*-d+1}\otimes\oplus_{\eta\in X_{(d)}}k(\eta)^\times,
\]
we have the exact sequence
\begin{equation}\label{eqn:PresentOmega}
\oplus_{\eta\in X_{(d)}} \L_{*-d+1}\otimes k(\eta)^\times \xrightarrow{\Div} 
 \Omega_{*}^{\prime(1)}(X)
 \xrightarrow{i_*}\Omega_*(X)\xrightarrow{j^*}\oplus_{\eta\in X_{(d)}} \Omega_*(k(\eta))\to 0
 \end{equation}
 and the extension of diagram \eqref{eqn:HM1} to the commutative diagram 
 \begin{equation}\label{eqn:HM2}
\xymatrixcolsep{12pt}
\xymatrix{
\oplus_{\eta} \L_{*-d+1}\otimes k(\eta)^\times\ar[r]^-{\Div}\ar@{=}[d]&\Omega_*^{(1)}(X)\ar[d]_{\vartheta^{(1)}}\ar[r]^-{i_*}&\Omega_*(X)\ar[d]_{\vartheta(X)}\ar[r]^-{j^*}&\oplus_{\eta}\Omega_*(k(\eta))\ar[r]\ar[d]^{\vartheta}&0\\
\oplus_{\eta} \L_{*-d+1}\otimes k(\eta)^\times\ar[r]_-{div_\MGL}&\MGL_{2*,*}^{\prime (1)}(X)\ar[r]_-{i_*}&\MGL_{2*,*}'(X)\ar[r]_-{j^*}&\oplus_{\eta}\MGL_{2*,*}'(k(\eta))\ar[r]&0}
\end{equation}
with exact rows and vertical arrows isomorphisms. Here $div_\MGL:=\del\circ\oplus_\eta t_\MGL(\eta)$.

\section{$K$-theory}

We apply the machinery described in  \S\ref{sec:Orient} to $K$-theory and connective $K$-theory, with the orientations given by remark~\ref{rem:Orient} (for $K$-theory) and lemma~\ref{lem:orient} (for connective $K$-theory). This gives us the  oriented duality theories $(\sK'_{*,*}, \sK^{*,*})$ and $(\sC\sK'_{*,*}, \sC\sK^{*,*})$; the canonical morphism $\sC\sK\to \sK$ is compatible with the orientations and thus gives the map of oriented duality theories $(\sC\sK'_{*,*}, \sC\sK^{*,*})\to (\sK'_{*,*}, \sK^{*,*})$.

We can in fact describe $\sC\sK'_{a,b}(X)$ more explicitly, with the help of the models $K^{(n)}$ for $f_nK$ and Quillen's localization theorem.

Indeed, for a smooth pair $(M,X)$, Quillen's localization theorem gives us the natural homotopy fiber sequence
\[
G(X)\to K(M)\to K(M\setminus X),
\]
where $G(X)$ is the Quillen-Waldhausen $K$-theory spectrum of the abelian category of coherent sheaves on $X$. This, together with ``algebraic Bott-periodicity" 
\[
\sK^{a,b}(X)\cong K_{2b-a}(X)
\]
gives us the canonical identifications
\[
\sK'_{a,b}(X)\cong G_{b-2a}(X);\ \sK^{a,b}_X(M)\cong K^X_{2b-a}(M)
\]
where $ K^X(M)$ is as usual the homotopy fiber of the restriction map $K(M)\to K(M\setminus X)$.

The construction of the simplicial spectrum $m\mapsto K^{(n)}(X,m)$ can be carried through replacing $K$-theory with $G$-theory, that is, the $K$-theory of the category of coherent sheaves. For details, we refer the reader to \cite[\S8.4]{TechLoc} and \cite{LevineAlgKthyI}; for the reader's convenience, we give a brief sketch of the construction here.

We denote the category of coherent sheaves on a scheme $Y$ by $\sM_Y$. Let  $X$ be a finite type $k$-scheme. There is a technical problem in the construction of the simplicial spectrum $m\mapsto G^{(n)}(X,m)$, due to the fact that, for $g:\Delta^p\to \Delta^m$ an inclusion of a face of $\Delta^m$, the corresponding restriction map
\[
(\id_X\times g)^*:\sM_{X\times\Delta^m}\to \sM_{X\times\Delta^p}
\] 
is not exact, hence does not lead to a functor on the corresponding $K$-theory spectra. To get around this problem, one replaces $\sM_{X\times\Delta^m}$ with the full subcategory  $\sM_{X\times\Delta^m,\del}$ of coherent sheaves which are Tor-independent with respect to the inclusions of faces; for $U\subset  X\times\Delta^m$ an open subset, we let $\sM_{U,\del}$ be the similarly defined full subcategory of $\sM_U$.  Letting $G(X,m)$ be the corresponding $K$-theory spectrum of the exact category  $\sM_{X\times\Delta^m,\del}$, we have the simplicial spectrum $m\mapsto G(X,m)$. 

Since we are not assuming that $X$ is equi-dimensional over $k$, it is more convenient to index by dimension rather than codimension.  Let  $\sS^X_n(m)$ denote the set of closed subsets $W$ of $X\times \Delta^m$ such that
\[
\dim W\cap X\times F\le n+p
\]
for all faces $F\cong \Delta^p$ of $\Delta^m$, and let
\[
G_{(n)}(X,m):=\hocolim_{W\in \sS^X_n(m)}G_W(X,m)
\]
where $G_W(X,m)$ is by definition the homotopy fiber of the restriction map 
\[
K(\sM_{X\times\Delta^m,\del})\to K(\sM_{X\times\Delta^m\setminus W,\del}).
\]
This gives us the simplicial spectrum $m\mapsto G_{(n)}(X,m)$ with total spectrum $G_{(n)}(X)$, and the corresponding tower of spectra
\[
\ldots\to G_{(n)}(X)\to G_{(n+1)}(X)\to\ldots\to G_{(\dim X)}(X)=G_{(\dim X+1)}(X,m)=\ldots
\]
The identification of $G(X):=K(\sM_X)$ with the 0-simplices in $G_{(\dim X)}(X)$ gives us the map of spectra
\[
G(X)\to G_{(\dim X)}(X),
\]
which, using the homotopy invariance property of $G$-theory, is easily seen to be a weak equivalence. 

It is easy to see that $X\mapsto G_{(n)}(X)$ is contravariantly functorial with respect to flat maps of finite type $k$-schemes. As for the usual $G$-theory construction, there is a canonical push-forward functor
\[
i_*:G_{(n)}(W)\to G_{(n)}(X)
\]
for $i:W\to X$ a closed immersion; this extends to a push forward map for arbitrary projective morphisms, but only in the homotopy category $\SH$ as one needs to replace the various categories of coherent sheaves with the corresponding subcategories which have no higher direct images with respect to the given projective morphism (see  \cite[\S 4]{LevineAlgKthyI} for details).

The localization theorem for the spectra $G_{(n)}$ from \cite[corollary 8.12]{TechLoc} is
\begin{thm} Let $j:U\to X$ be an open immersion of finite type $k$-schemes with closed complement $i:W\to X$. Then the sequence
\[
G_{(n)}(W)\xrightarrow{i_*} G_{(n)}(X)\xrightarrow{j^*} G_{(n)}(U)
\]
is a  homotopy fiber sequence.
\end{thm}

\begin{cor}\label{cor:Computation2} Let $X$ be  in $\Sch/k$.\\\\
1. There is a canonical isomorphism
\[
a^X_{a,b}:\sC\sK'_{a,b}(X)\to\pi_{a-2b}(G_{(b)}(X))
\]
compatible with  push-forward for projective morphisms and pull-back for open immersions.
In particular, we have isomorphisms
\[
a^X_{2b,b}: \sC\sK'_{2b,b}(X)\cong\pi_{0}(G_{(b)}(X))
\]
for all $b\in\Z$, compatible with  push-forward for projective morphisms and pull-back for open immersions.  \\\\
2. For $X\in\Sm/k$ of dimension $d$, the isomorphism $a^X_{a,b}$ is compatible with the isomorphism of proposition~\ref{prop:Comp1} via the weak equivalence
\[
G_{(b)}(X))\to K^{(d-b)}(X))
\]
and the isomorphism $\alpha_X:\sC\sK^{2d-a,d-b}(X)\to \sC\sK'_{a,b}(X)$.
\end{cor}

\begin{proof}  For each $q$, let $K^{(q)}_X(M)$ denote the homotopy fiber of the restriction map
\[
j^*:K^{(q)}(M)\to K^{(q)}(M\setminus X).
\]
Using example~\ref{ex:Supports} and proposition~\ref{prop:Comp1} for the diagram
\[
\xymatrix{
M\setminus X\ar[r]\ar[d]&M\\
{*}}
\]
gives us the canonical isomorphism
\[
\sC\sK'_{a,b}(X)\xymatrix{\ar[r]^{\alpha_{M,X}}_\sim&} \sC\sK^{2m-a,m-b}_X(M)\cong \pi_{a-2b}(K^{(m-b)}_X(M)).
\]
By the localization theorem, we have the canonical weak equivalence
\[
\beta_{X,M}:G_{(b)}(X)\to \fib(G_{(b)}(M)\to G_{(b)}(M\setminus X))
\]
induced by the functors $(i_X\times\id)_*:\sM_{X\times\Delta^n}\to \sM_{M\times\Delta^n}$ and the canonical isomorphism of $(j_{M\setminus X}\times\id)^*\circ(i\times\id)_*$ with the 0-functor. The resolution theorem gives a canonical weak equivalence
\begin{multline*}
K^{(m-b)}_X(M)= \fib(K^{(m-b)}(M)\to K^{(m-b)}(M\setminus X))\\
\xrightarrow{\gamma_{X,M}}  \fib(G_{(b)}(M)\to G_{(b)}(M\setminus X)).
\end{multline*}
This defines the isomorphism in $\SH$
\[
\delta_{X,M}:=\beta_{X,M}^{-1}\circ\gamma_{X,M}:K^{(m-b)}_X(M)\to G_{(b)}(X).
\]

Both $\gamma_{X,M}$ and $\beta_{X,M}$ are compatible with pull-back by flat maps $f:M'\to M$ (with $X':=f^{-1}(X)$), hence $\delta$ is compatible with pull-back by open immersions. 

Suppose we have a closed immersion $i:X'\to X$. We use the same ambient smooth scheme $M$, giving the map of pairs $\id_M:(X, M)\to(X',M)$. The push-forward by $i$ on $K^{(m-b)}_X(M)$ is by definition the map on the homotopy fibers induced by the commutative diagram
\[
\xymatrix{
K^{(m-b)}(M)\ar[r]^-{u^*} \ar[d]_{\id}&K^{(m-b)}(M\setminus X'))\ar[d]^{w^*}\\
K^{(m-b)}(M)\ar[r]^-{v^*} &K^{(m-b)}(M\setminus X'))
}
\]
where $u, v, w$ are the evident open immersions. It is not hard to check that the diagram
\[
\xymatrix{
G_{(b)}(X')\ar[r]^-{i_{X'*}}\ar[d]_{i_*}& G_{(b)}(M)\ar[r]^-{u^*}\ar[d]_{\id}& G_{(b)}(M\setminus X'))\ar[d]^{w^*}\\
G_{(b)}(X)\ar[r]_-{i_{X*}}& G_{(b)}(M)\ar[r]_-{v^*}& G_{(b)}(M\setminus X))
}
\]
commutes up to canonical homotopy. From this it follows that $i^G_*\circ \delta_{X,M}=\delta_{X',M}\circ i^K_*$. 

For both $K$-theory and $G$-theory, the push-forward for a projection $p:Y\times\P^n\to Y$ is compatible with the projection onto the factor $[\sO_{Y\times\P^n}]$ in the respective projective bundle formulas with bases $[\sO(-i)]$, $i=0,\ldots, n$. From this it is not hard to show that $\delta_{X,M}$ is compatible with $p_*$, and hence compatible with push-forward by an arbitrary projective morphism, completing the proof of (1).

The assertion (2) follows directly from the construction of $\delta_{X,M}$.
\end{proof}

\begin{rem} Of course, Quillen's localization theorem tells us that the homology theory $\sK'_{*,*}$ associated to $K$-theory is just $G$-theory. It follows from corollary~\ref{cor:Computation2} that the canonical natural transformation $\sC\sK'_{a,b}\to \sK'_{a,b}$ is given by the  map
\[
\pi_{a-2b}G_{(b)}(X)\to \pi_{a-2b}G(X)=G_{a-2b}(X)
\]
induced by the canonical map $G_{(b)}X()\to G_{(dim X)}(X)\sim G(X)$.
\end{rem}

Our main interest is in the ``geometric" portion $\sC\sK'_{2*,*}$ of the theory $\sC\sK'_{*,*}$. For a finite type $k$-scheme $X$, and integer $b\ge0$, we have the full subcategory $\sM_{(b)}(X)$ of $\sM_X$ with objects the coherent sheaves $\sF$ on $X$ with $\dim_k\supp \sF\ge b$. The following result improves upon corollary~\ref{cor:CKComp}:

\begin{thm}\label{thm:ConnK0} Let $X$ be a finite type $k$-scheme. There is a canonical isomorphism
\[
a^X_n:\sC\sK'_{2n,n}(X)\to \im(K_0(\sM_{(n)}(X)\to \sM_{(n+1)}(X)),
\]
compatible with pull-back by open immersions and push-forward by projective morphisms. Furthermore $\sC\sK'_{a,b}(X)=0$ for $a<2b$.
\end{thm}

\begin{proof} By corollary~\ref{cor:Computation2}, we have
\[
\sC\sK'_{a,b}(X)\cong \pi_{a-2b}(G_{(b)}(X)).
\]
Moreover, $G_{(b)}(X)$ is the simplicial spectrum $m\mapsto G_{(b)}(X,m)$ and 
\[
G_{(b)}(X,m)=\hocolim_{W\in\sS_{(b)}^X(m)}G_W(\sM(X\times\Delta^m,\del)).
\]
We showed in \cite[lemma 8.7]{TechLoc} that for all open $U\subset X\times\Delta^m$, the inclusion
\[
\sM(U,\del)\to \sM_U
\]
induces a weak equivalence on the $K$-theory spectra. Thus $G_W(\sM(X\times\Delta^m,\del))$ is weakly equivalent to the homotopy fiber of the restriction map
\[
j^*:G(X\times\Delta^m)\to G(X\times\Delta^m\setminus W)
\]
which by Quillen's localization theorem is in turn weakly equivalent to $G(W)$. In particular, this shows that $G_{(b)}(X,m)$ is -1 connected and we have a strongly convergent spectral sequence
\[
E^1_{p,q}=\pi_{p+q}G_{(b)}(X,p)\Longrightarrow \pi_{p+q}G_{(b)}(X).
\]
As $E^1_{p,q}=0$ for $p+q<0$, this shows that $\sC\sK'_{a,b}(X)=0$ for $a<2b$, as claimed.

For the assertion on $\sC\sK'_{2n,n}(X)$, the spectral sequence gives the right exact sequence
\[
\pi_0(G_{(n)}(X,1))\xrightarrow{\del}\pi_0(G_{(n)}(X,0))\to \sC\sK'_{2n,n}(X)\to0.
\]
By definition, $G_{(n)}(X,0))=K(\sM_{(n)}(X))$, so we need to show that the image $M$ of $\pi_0(G_{(n)}(X,1))$ under $\del$ is the same as the kernel  $N$ of the map
\[
K_0(\sM_{(n)}(X))\to K_0(\sM_{(n+1)}(X)).
\]

We note that $\sM_{(n)}(X)$ is a Serre subcategory of the abelian category $\sM_{(n+1)}(X)$; let $\sM_{(n+1/n)}(X)$ denote the quotient category. By Quillen's localization theorem, we have the exact sequence
\[
\ldots\to K_1(\sM_{(n+1/n)}(X))\xrightarrow{\del} K_0(\sM_{(n)}(X))\to K_0(\sM_{(n+1)}(X))\to \ldots .
\]
Furthermore, by devissage, we have
\[
K_1(\sM_{(n+1/n)}(X))\cong\oplus_{x\in X_{(n)}}K_1(k(x))=\oplus_{x\in X_{(n)}}k(x)^\times
\]
where $X_{(n)}$ is the set of dimension $n$ points of $X$. Finally, the boundary map $K_1(\sM_{(n+1/n)}(X))\to K_0(\sM_{(n)}(X))$ can be described as follows: Take $y\in X_{(n+1)}$ and let $Y\subset X$ be the reduced closure of $y$. Take $f\in k(y)^\times$ and let $Y'\subset Y\times\P^1$ be the reduced closure of the graph of $f$.  For $t\in \P^1(k)$,  let $\sO_{Y'}(t):=\sO_{Y'}\otimes_{\sO_{\P^1}}k(t)$. Then
\[
\del(f\in k(y)^\times)=p_{1*}([\sO_{Y'}(0)]-[\sO_{Y'}(\infty)]).
\]
Noting that $\sO_{Y'}(t)$ is an $\sO_{X\times t}$-module, we can consider $p_{1*}([\sO_{Y'}(0)])$ as simply 
$[\sO_{Y'}(0)]$ via the isomorphism $p_1:X\times0\to X$, and similarly for $[\sO_{Y'}(\infty)]$. Finally, as $f$ is in $k(y)^\times$ and $Y$ has dimension $n+1$, it follows that $\sO_{Y'}(0)$ and $\sO_{Y'}(\infty)$ have support of dimension at most $n$. 

Given such a $Y$ and $f\in k(y)^\times$, we may identify $X\times\Delta^1$ with the open subscheme $X\times(\P^1\setminus\{1\})$ of $X\times\P^1$ via the open immersion $j:\Delta^1\to \P^1$ defined by  $j(t_0, t_1):=t_1/t_0$. Then $(\id\times j)^{-1}(\sO_{Y'})$ is a coherent sheaf in $\sM_{(n)}(X\times\Delta^1,\del)$ and 
\[
\del([(\id\times j)^{-1}(\sO_{Y'})])=\pm\del(f\in k(y)^\times).
\]
Thus, $M\supset N$. 

For the reverse inclusion, take $\sF\in \sM_{(n)}(X\times\Delta^1,\del)$, let $Z$ be the support of $\sF$, let $Y\subset X$ be the closure of $p_1(Z)$ and let $W\subset X\times\Delta^1$ be $p_1^{-1}(Y)$; let $i:Z\to W$ be the inclusion. By devissage, we may assume that $\sF$ is an $\sO_Z$-module. But $W\cong Y\times\A^1$, so by the homotopy property for $G$-theory, there is a class $\eta\in G_0(Y)$ with 
\[
i_*([\sF])=p_1^*(\eta).
\]
Letting $i_0, i_1:Y\to Y\times\Delta^1=W$ be the inclusions $t_0=0$, $t_1=0$, respectively, this shows that
\[
i_0^*(i_*([\sF]))-i_1^*(i_*([\sF]))=0
\]
in $G_0(Y)$. As $Y$ has dimension $\le n+1$, this shows that the image of $\del([\sF])$ in $G_0(\sM_{(n+1)}(X))$ is zero, and hence $M\subset N$.
\end{proof}

\begin{cor} \label{cor:Injectivity} Take $X\in\Sch/k$ and let $d$ be an integer with $\dim_kX\le d$. Then the natural map
\[
\gamma^X_n:\sC\sK'_{2n,n}(X)\to \sK'_{2n,n}(X)=G_0(X)
\]
is an isomorphism for $n\ge d$ and is injective for $n=d-1$.
\end{cor}

\begin{proof} The identity $\sK'_{2n,n}(X)=G_0(X)$ follows from the fact that the $T$-spectrum $\sK$ represents Quillen $K$-theory on $\Sm/k$ and Quillen's localization theorem identifies the homotopy fiber of $K(M)\to K(M\setminus X)$ with the $G$-theory spectrum $G(X)$, for $M\in \Sm/k$ and $i:X\to M$ a closed immersion. 

For $n\ge d-1$, theorem~\ref{thm:ConnK0} gives the identification
\[
 \pi_0(G_{(n)}(X))\cong \sC\sK'_{2n,n}(X)\cong im(K_0(\sM_{(n)}(X))\to G_0(X)),
\] 
the first isomorphism being induced by the localization fiber sequence
\[
G_{(n)}(X))\to K_{(n)}(M))\xrightarrow{j^*} K_{(n)}(M\setminus X)
\]
for a given closed immersion $i:X\to M$, $M\in \Sm/k$, and the second isomorphism the computation of $\sC\sK'_{2n,n}(X)$ given in theorem~\ref{thm:ConnK0}. We need only check that via this second isomorphism and the identity $\sK'_{2n,n}(X)=G_0(X)$, the canonical map $\rho:\sC\sK'_{2n,n}(X)\to \sK'_{2n,n}(X)$ transforms to the evident map $im(K_0(\sM_{(n)}(X))\to G_0(X))\to G_0(X)$, which is clearly injective for all $n$, and an isomorphism for $n\ge d$.

We have the commutative square
\[
\xymatrix{
K_{(n)}(M))\ar[d]\ar[r]^-{j^*}& K_{(n)}(M\setminus X)\ar[d]\\
K_{(d')}(M)\ar[r]_-{j^*}&K_{(d')}(M\setminus X)
}
\]
where $d'$ is any integer with $d'\ge n$, $d'\ge \dim_kM$. We have as well canonical weak equivalences $K(M)\to K_{(d')}(M)$, $K(M\setminus X)\to K_{(d')}(M\setminus X)$, induced by the identification of $K(M)$ with the 0-spectrum in the simplicial spectrum $m\mapsto K_{(d')}(M,m)$, and similarly for $M\setminus X$. This shows that  the natural map $\sC\sK'_{2n,n}(X)\to \sK'_{2n,n}(X)$ is given by the canonical map
\[
\pi_0(G_{(n)}(X))\to \pi_0(G_{(d')}(X))=G_0(X).
\]
Via the computation of  theorem~\ref{thm:ConnK0}, this is just the evident map
\[
im(K_0(\sM_{(n)}(X))\to G_0(X))\to G_0(X),
\]
as desired.
\end{proof}

\begin{rems}\label{rem:Cai} 1. For $X$ smooth of dimension $d$, we may take the identity closed immersion to define $\sC\sK'_{**}(X)$, so that
\[
\sC\sK'_{a,b}(X)=\sC\sK^{2d-a,d-b}(X)
\]
In particular, letting $\sM^{(b)}(X)\subset \sM_X$ be the full subcategory of coherent sheaves on $X$ with support in codimension at least $b$, theorem~\ref{thm:ConnK0} shows that
\[
\sC\sK^{2n,n}(X)=\im(K_0(\sM^{(n)}(X)\to K_0(\sM^{(n-1)}(X)).
\]
In fact, the argument for  theorem~\ref{thm:ConnK0} works to show this identity for $k$ an arbitrary  perfect field, we do not need resolution  of singularities here. 
\\\\
2. Cai \cite{Cai} has defined a theory of ``connective higher $K$-theory" as a bi-graded oriented cohomology theory on $\Sm/k$. Denoting this theory as $CK^{a,b}$, Cai defines $CK^{a,b}$ as
\[
CK^{a,b}(X):=\im(K_{2b-a}(\sM^{(b)}(X)\to \sM^{(b-1)}(X))
\]
Thus we have 
\[
\sC\sK^{2n,n}(X)=CK^{2n,n}(X)
\]
for $X\in \Sm/k$.
\end{rems}

We conclude this section with a description of the 1st Chern class operators on $\sC\sK_{2n,n}'(X)$. Let $L$ be a line bundle on $X$, and let $i:X\to M$ be a closed immersion with  $M\in\Sm/k$ such that $L$ extends to a line bundle $\sL$ on $M$. We note that the 1st Chern class operator $\tilde{c}_1(L)$ on $\sC\sK_{2n,n}'(X)$ is given by the product with $c_1(\sL)\in \sC\sK^{2,1}(M)$ on $\sC\sK^{2*,*}_X(M)$. As $\sC\sK^{2,1}(M)\to \sK^{2,1}(M)=K_0(M)$ is injective (corollary~\ref{cor:Injectivity}), it follows that the operators $\tilde{c}_1(L)$ satisfy the multiplicative formal group law:
\[
\tilde{c}_1(L\otimes L')=\tilde{c}_1(L)+\tilde{c}_1(L')-\beta\tilde{c}_1(L)\circ \tilde{c}_1(L')
\]
where $\beta\in \sC\sK_{2,1}'(k)$ is the ``Bott element", i.e., the element corresponding to $1\in\Z$ under the sequence of isomorphisms
\[
\sC\sK_{2,1}'(k)=\sC\sK^{-2,-1}(k)\to \sK^{-2,-1}(k)\cong K_0(k)\xrightarrow{\dim_k}\Z.
\]
Thus, we need only describe $\tilde{c}_1(L)$ for $L$ which is very ample on $X$. 

Let $\alpha$ be in $\sC\sK_{2n,n}'(X)$ for some $n$. If $n\ge\dim_kX+1$, then both $\sC\sK_{2n,n}'(X)$ and $\sC\sK_{2n-2,n-1}'(X)$ are equal to $G_0(X)$ via the canonical map $\sC\sK_{2*,*}'(X)\to \sK_{2*,*}'(X)=G_0(X)$, so $\tilde{c}_1(L)$ is given by the 1st Chern class operator on $\sK_{2*,*}'(X)=G_0(X)$. This in turn is given by multiplication by the $K$-theory 1st Chern class via the $K_0(X)$-module structure on $G_0(X)$, that is, 
\[
\tilde{c}^{G_0}_1(L)(\alpha)=(1-L^{-1})\cdot\alpha.
\]

In general, we have our description 
\[
\sC\sK_{2n,n}'(X)\cong  im(K_0(\sM_{(n)}(X))\to K_0(\sM_{(n+1)}(X))), 
\]
compatible with projective push-forward.  Thus, for each $\alpha\in \sC\sK_{2n,n}'(X)$, there is a closed immersion $i:X'\to X$ in $\Sch/k$, with $\dim_kX'\le n$, and an element $\alpha'\in 
 \sC\sK_{2n,n}'(X')$ with $i_*(\alpha')=\alpha$. Since
 \[
 \tilde{c}_1(L)(i_*\alpha')=i_*(\tilde{c}_1(i^*L)(\alpha')),
 \]
 we reduce to the case $n\ge \dim_kX$; similarly, we may assume $X$ is irreducible. By our above computation, the only remaining case is $n= \dim_kX$.
 
In this case, $\sC\sK_{2n,n}'(X)=G_0(X)$ and for $\alpha\in \sC\sK_{2n,n}'(X)$, $\tilde{c}_1(\alpha)\in \sC\sK_{2n-2,n-1}'(X)$ is characterised by the identity
\[
(1-L^{-1})\cdot \alpha=im(\tilde{c}_1(\alpha)\in G_0(X)),
\]
via the inclusion (corollary~\ref{cor:Injectivity}) $\sC\sK_{2n-2,n-1}'(X)\to G_0(X)$. We now use the assumption that $L$ is very ample. Write $\alpha=\sum_in_i[\sF_i]$, where the $\sF_i$ are coherent sheaves on $X$, $n_i$ are integers and $[-]$ denotes class in $G_0(X)$. Choose a   section $s$ of $L$ that is locally a non-zero divisor on each $\sF_i$ and on $\sO_X$; this is possible by Kleiman's transversality theorem and the fact that $L$ is very ample. Let $H\subset X$ be the closed subschema defined by $s$, giving us the short exact sequence of sheaves on $X$
\[
0\to L^{-1}\to \sO_X\to \sO_H\to 0
\]
For each $i$, the tensor product $\sF_i\otimes\sO_H$ is an element of $\sM_{(n-1)}(X)$ (its support is contained in $H$), giving the well-defined class $[\sF_i\otimes\sO_H]$ in $K_0(\sM_{(n-1)}(X))$. In addition, the above exact sequence and the fact that $s$ is a non-zero divisor on $\sF_i$ gives the identity
\[
[\sF_i\otimes\sO_H]=(1-L^{-1})\cdot[\sF_i]\in G_0(X)
\]
for each $i$ and hence $\tilde{c}_1(\alpha)\in \sC\sK_{2n-2,n-1}'(X)$ is given by $\sum_in_i[\sF_i\otimes\sO_H]$.

\begin{rem}\label{rem:ConnGThy} Set $CG_{a,b}(X):= \pi_{a-2b}G_{(b)}(X)$, for $X\in \Sch/k$. From the proof of corollary~\ref{cor:Computation2} we have isomorphisms $\delta_{X,M}:\pi_{a-2b}(K^{(m-b)}_X(M))\to \pi_{a-2b}(G_{(b)}(X))$ for $X$ a closed subscheme of $M\in \Sm/k$, $m=\dim M$. Combined with the natural isomorphisms 
$\sC\sK_X^{2m-a,m-b}(M)\cong \pi_{a-2b}(K^{(m-b)}_X(M))$ given by theorem~\ref{thm:HC} and  lemma~\ref{lem:Rep}, the properties of $G_{(b)}(-)$ discussed in this section show that we have defined an oriented duality theory  $(CG_{*,*}, \sC\sK^{*,*})$ on $\Sch/k$, for $k$ an arbitrary perfect field, except that we have nt defined a cap product structure of $ \sC\sK^{*,*}$ on $CG_{*,*}$. Possibly this could be supplied by the method used by Cai in \cite[\S 6.3]{Cai} to define pull back maps in his theory $CK^{*,*}$ for regular embeddings.

The results of this section can be interpreted as saying that, in case $k$ admits resolution of singularities, the oriented duality theory $(\sC\sK'_{*,*}, \sC\sK^{*,*})$ on $\Sch/k$ defined using the results of \cite{LevineOrient} is isomorphic to $(CG_{*,*}, \sC\sK^{*,*})$ (neglecting the cap product structure).
\end{rem}

\section{The comparison map} We consider the classifying map
\[
\theta_{\sC\sK}:\Omega_*\to \sC\sK'_{2*,*}
\]
constructed in proposition~\ref{prop:Universal}. Let $\phi_{CK}:\L_*\to \Z[\beta]$ be the ring homomorphism classifying the formal group law $(u+v-\beta uv, \Z[\beta])$ of the theory $\sC\sK$ and $\phi_\Omega:\L_*\to \Omega_*(k)$ the classifying map for the formal group law for $\Omega_*$. For each $X\in\Sch/k$, the external products make $\Omega_*(X)$ into an $\Omega_*(k)$-module and thus via $\phi_\Omega$ an $\L_*$-module. It is easy to check that the various structures for $\Omega_*$ make the assignment $X\mapsto \Omega_*(X)\otimes_{\L_*}\Z[\beta]$ into an oriented Borel-Moore homology theory on $\Sch/k$, which we denote by $\Omega_*\otimes_{\L_*}\Z[\beta]$; the universality of $\Omega_*$ makes $\Omega_*\otimes_{\L_*}\Z[\beta]$ the universal oriented Borel-Moore homology theory on $\Sch/k$ having formal group law $(u+v-\beta uv, \Z[\beta])$.

\begin{lem} \label{lem:Descent} $\theta_{\sC\sK}$ factors through the surjection
\[
\Omega_*\to \Omega_*\otimes_{\L_*}\Z[\beta].
\]
\end{lem}

\begin{proof}  As the natural transformation $\theta_{\sC\sK}$ is compatible with external products, $\theta_{\sC\sK}(k)$ is a ring homomorphism and $\theta_{\sC\sK}(X)$ is a map of $\L_*-\Z[\beta]$-modules. Using the universal property of $\Omega_*\otimes_{\L_*}\Z[\beta]$, we need only show that
\[
\theta_{\sC\sK}(k)\circ\phi_\Omega:\L_*\to \Z[\beta]
\]
is  equal to the classifying homomorphism $\phi_{CK}$. But this is a general property of the classifying map $\theta_{H,A}:\Omega_*\to H_*$ for any oriented duality theory $(H,A)$, as we have already noted in remark~\ref{rem:CompClassFGL}. 
\end{proof}

In \cite[theorem 1.2.18]{LevineMorel} it is shown that the classifying map of oriented cohomology theories on $\Sm/k$
\[
\vartheta_{K_0}:\Omega^*\otimes_\L\Z[\beta,\beta^{-1}]\to K_0[\beta,\beta^{-1}]
\]
is an isomorphism, where $\L\to \Z[\beta,\beta^{-1}]$ is the classifying map for the formal group law $(u+v-\beta uv,  \Z[\beta,\beta^{-1}])$. Dai extended this result, showing in \cite{Dai} that the classifying map
\[
\vartheta_{G_0}:\Omega_*\otimes_\L\Z[\beta,\beta^{-1}]\to G_0[\beta,\beta^{-1}]
\]
is an isomorphism of  oriented Borel-Moore homology theories on $\Sch/k$. We now prove the analogous result for connective $K$-theory:

\begin{definition} For $X\in\Sch/k$, write $\Omega^{CK}_*(X)$ for $\Omega_*(X)\otimes_{\L_*}\Z[\beta]$. 
\end{definition}

\begin{thm} \label{thm:ClassIso} The induced classifying map
\[
\theta_{\sC\sK}:\Omega^{CK}_*\to \sC\sK'_{2*,*}
\]
is an isomorphism.
\end{thm}

\begin{proof}  For $X\in\Sch/k$, we consider $\theta_{\sC\sK}$ as a map of presheaves on $X_\Zar$, so we may evaluate at the generic points of $X$. For $\eta$ a dimension $d$ generic point of $X$, it follows from the description of $\sC\sK'_{2*,*}$ given by theorem~\ref{thm:ConnK0} that
\[
\sC\sK'_{2n,n}(\eta)=\begin{cases}K_0(k(\eta))&\text{ for }n\ge d\\ 0&\text{ for }n<d.\end{cases}
\]
By \cite[corollary 4.4.3]{LevineMorel}, the pull-back map $p_X^*:\Omega_*(k)\to \Omega_{*+d}(k(\eta))$ is an isomorphism, so by \cite[theorem 1.2.7]{LevineMorel} we have the isomorphism of $\L_*$-modules
\[
\Omega_{*}(k(\eta))\cong \L_{*-d}.
\]
From this it follows easily that 
\[
\theta_{\sC\sK}(k(\eta)):\Omega^{CK}_*(k(\eta))\to \sC\sK'_{2*,*}(k(\eta))
\]
is an isomorphism. In particular, $\theta_{\sC\sK}(X)$ is an isomorphism for all reduced $X$ of dimension 0 over $k$. As the inclusion $X_{red}\to X$ induces an isomorphism on $\Omega_*$ and $\sC\sK'_{2*,*}$, the theorem is proven for all $X$ of dimension 0. 

We proceed by induction on the maximum $d$ of the dimensions of the components of $X$; we may assume that $X$ is reduced.
We use the constructions and notations from theorem~\ref{thm:MGLComp} and the discussion following that theorem.  We let $\sC\sK_{2*,*}^{\prime(1)}(X)$ be the inductive limit
\[
\sC\sK_{2*,*}^{\prime (1)}(X):=\colim_W \sC\sK_{2*,*}'(W)
\]
as $W$ runs over all (reduced) closed subschemes of $X$ which contain no dimension $d$ generic point of $X$. This, together with the map $\Div$ defined following theorem~\ref{thm:MGLComp}, and the localization exact sequence for $\sC\sK_{*,*}'$ gives us the commutative diagram with exact rows
\begin{equation}\label{eqn:Diag1}
\xymatrixcolsep{10pt}
\xymatrix{
\oplus_{\eta\in X_{(d)}}\L_{*-d+1}\otimes k(\eta)^\times\ar[r]^-{\Div}&\Omega_*^{(1)}(X)\ar[d]_{\vartheta^{(1)}}\ar[r]^-{i_*}&\Omega_*(X)\ar[d]_{\vartheta(X)}\ar[r]^-{j^*}&\oplus_{\eta\in X_{(d)}}\Omega_*(k(\eta))\ar[r]\ar[d]^{\vartheta}&0\\
\oplus_{\eta\in X_{(d)}}\sC\sK_{2*+1,*}'(k(\eta))\ar[r]_-{\partial} &\sC\sK_{2*,*}^{\prime (1)}(X)\ar[r]_-{i_*}&\sC\sK_{2*,*}'(X)\ar[r]_-{j^*}&\oplus_{\eta\in X_{(d)}}\sC\sK_{2*,*}'(k(\eta))\ar[r]&0}
\end{equation}

We apply $\Z[\beta]\otimes_\L(-)$ to the top row in \eqref{eqn:Diag1}. By lemma~\ref{lem:Descent}, the vertical maps in \eqref{eqn:Diag1}  descend to give the commutative diagram
\begin{equation}\label{eqn:Diag2}
\xymatrixcolsep{8pt}
\xymatrix{
\oplus_{\eta}\Z[\beta]_{*-d+1}\otimes k(\eta)^\times\ar[rr]^-{\Div_{CK}}&&\Omega_*^{CK(1)}(X)\ar[d]_{\vartheta^{(1)}}\ar[r]^-{i_*}&\Omega^{CK}_*(X)\ar[d]_{\vartheta(X)}\ar[r]^-{j^*}&\oplus_{\eta\in X_{(d)}}\Omega^{CK}_*(k(\eta))\ar[r]\ar[d]^{\vartheta}&0\\
\oplus_{\eta}\sC\sK_{2*+1,*}'(k(\eta))\ar[rr]_-{\partial}&&\sC\sK_{2*,*}^{\prime (1)}(X)\ar[r]_-{i_*}&\sC\sK_{2*,*}'(X)\ar[r]_-{j^*}&\oplus_{\eta\in X_{(d)}}\sC\sK_{2*,*}'(k(\eta))\ar[r]&0}
\end{equation}

By induction on $d$, the map $\vartheta^{(1)}$ is an isomorphism; we have already seen that $\vartheta$ is an isomorphism. We note that the bottom row is a sequence of $\Z[\beta]$-modules via the isomorphism $\sC\sK_{2*,*}(k)\cong\Z[\beta]$ and the $\sC\sK_{2*,*}(k)$-module structure given by external products.

Take $\eta\in X_{(d)}$. Then 
$\sC\sK_{2*+1,*}'(k(\eta))\cong \sC\sK^{2d-2*-1,d-*}(k(\eta))$. Via lemma~\ref{lem:unit}, we have 
 the isomorphism of $\Z[\beta]$-modules
\[
t_{\sC\sK}:\Z[\beta]_{*-d+1}\otimes k(\eta)^\times.\to \sC\sK_{2*+1,*}'(k(\eta)).
\]
Putting this into the diagram \eqref{eqn:Diag2} gives us the commutative diagram
\begin{equation}\label{eqn:Diag3}
\xymatrixcolsep{8pt}
\xymatrix{
\oplus_{\eta}\Z[\beta]_{*-d+1}\otimes k(\eta)^\times\ar[rr]^-{\Div_{CK}}&&\Omega_*^{CK(1)}(X)\ar[d]_{\vartheta^{(1)}}\ar[r]^-{i_*}&\Omega^{CK}_*(X)\ar[d]_{\vartheta(X)}\ar[r]^-{j^*}&\oplus_{\eta\in X_{(d)}}\Omega^{CK}_*(k(\eta))\ar[r]\ar[d]^{\vartheta}&0\\
\oplus_{\eta}\Z[\beta]_{*-d+1}\otimes k(\eta)^\times\ar[rr]_-{div_{CK}}&&\sC\sK_{2*,*}^{\prime (1)}(X)\ar[r]_-{i_*}&\sC\sK_{2*,*}'(X)\ar[r]_-{j^*}&\oplus_{\eta\in X_{(d)}}\sC\sK_{2*,*}'(k(\eta))\ar[r]&0}
\end{equation}
with the bottom row exact and the top row a complex.
 
We claim the identity map $Z[\beta]\otimes k(\eta)^\times\to Z[\beta]\otimes k(\eta)^\times$ fills in the diagram \eqref{eqn:Diag3} to a (up to sign) commutative diagram. Assuming this claim, it follows by a diagram chase that the top row is exact and the map $\vartheta(X)$ is an isomorphism. 

To prove the claim, the orientation $c_{\sC\sK}$ for $\sC\sK$ and the universal property of $\MGL$ gives the canonical map of oriented cohomology theories
\[
\rho_{\sC\sK}:(\MGL, c_\MGL)\to (\sC\sK, c_{\sC\sK})
\]
which in turn gives us the map of bi-graded oriented duality theories
\[
\rho_{\sC\sK}:(\MGL'_{*,*}, \MGL^{*,*})\to (\sC\sK'_{*,*}, \sC\sK^{*,*}).
\]
It follows from the characterization of $\theta_\MGL$, $\theta_{\sC\sK}$ given in remark~\ref{rem:Char} that
\[
\theta_{\sC\sK}=\rho_{\sC\sK}\circ\theta_\MGL.
\]

As discussed at the end of \S\ref{sec:ExplicitModel} and in lemma~\ref{lem:unit},   the orientations $c_\MGL$, $c_{\sC\sK}$ give rise  to canonical elements
\[
t_\MGL\in \MGL^{1,1}(\G_m),\ t_{\sC\sK}\in \sC\sK^{1,1}(\G_m).\
\]
These in turn give by functoriality canonical homomorphisms for each $X\in \Sm/k$
\[
t_\MGL(X):\sO_X^\times(X)\to  \MGL^{1,1}(X),\ t_{\sC\sK}(X):\sO_X^\times(X)\to  \sC\sK^{1,1}(X)
\]
with $t_{\sC\sK}(X)=\rho_{\sC\sK}(X)\circ t_{\sC\sK}(X)$. We extend $t_\MGL$ to 
\[
t_\MGL:\sO_X^\times(X)\otimes\L^*\to \MGL^{2*+1,*+1}(X)
\]
using the $\L^*$-module structure, and similarly have the extension of $ t_{\sC\sK}(X)$ to 
\[
t_{\sC\sK}:\sO_X^\times(X)\otimes\Z[\beta]^*\to \sC\sK^{2*+1,*+1}(X).
\]

The map $k(\eta)^\times\to \MGL^{1,1}(k(\eta))$ arising in the Hopkins-Morel spectral sequence is the map $t_\MGL(k(\eta))$. As $\rho_{\sC\sK}$ is a map of $\L_*$-$\Z[\beta]$ modules, we have the commutative diagram
\[
\xymatrix{
\oplus_{\eta\in X_{(d)}} \L_{*-d+1}\otimes k(\eta)^\times\ar[r]^-{t_\MGL}\ar[d]_\pi&
\oplus_{\eta\in X_{(d)}}\MGL'_{2*+1,*}(k(\eta))\ar[r]^-{\partial}\ar[d]_{\rho_{\sC\sK}}
&\MGL_{2*,*}^{\prime (1)}(X)\ar[d]^{\rho_{\sC\sK}(X^{(1)}}\\
\oplus_{\eta\in X_{(d)}} \Z[\beta]_{*-d+1}\otimes k(\eta)^\times\ar[r]_-{t_{\sC\sK}}&\sC\sK'_{2*+1,*}\ar[r]_{\partial}&\sC\sK_{2*,*}^{\prime (1)}(X)
}
\]
where $\pi$ is induced by the classifying map $\L_*\to\Z[\beta]$. The map $div_\MGL$ in diagram \eqref{eqn:HM2} is the composition $\partial\circ t_\MGL$ in the diagram above. Defining 
\[
div_{CK}:\oplus_{\eta\in X_{(d)}} \Z[\beta]_{*-d+1}\otimes k(\eta)^\times\to \sC\sK_{2*,*}^{\prime (1)}(X)
\]
as $div_{CK}:=\partial\circ t_{\sC\sK}$ gives us the commutative diagram
\[
\xymatrix{
\oplus_{\eta\in X_{(d)}} \L_{*-d+1}\otimes k(\eta)^\times\ar[r]^-{div_\MGL}\ar[d]_\pi
&\MGL_{2*,*}^{\prime (1)}(X)\ar[d]^{\rho_{\sC\sK}(X^{(1)}}\\
\oplus_{\eta\in X_{(d)}} \Z[\beta]_{*-d+1}\otimes k(\eta)^\times\ar[r]_-{div_{CK}}&\sC\sK_{2*,*}^{\prime (1)}(X)
}
\]
patching in the left-hand square in the commutative diagram \eqref{eqn:HM2} yields the commutative diagram
\begin{equation}\label{eqn:HM3}
\xymatrix{
\oplus_{\eta} \L_{*-d+1}\otimes k(\eta)^\times\ar[r]^-{\Div}\ar@{=}[d]&\Omega_*^{(1)}(X)\ar[d]_{\vartheta^{(1)}}\\
\oplus_{\eta\in X_{(d)}} \L_{*-d+1}\otimes k(\eta)^\times\ar[r]^-{div_\MGL}\ar[d]_\pi
&\MGL_{2*,*}^{\prime (1)}(X)\ar[d]^{\rho_{\sC\sK}(X^{(1)}}\\
\oplus_{\eta\in X_{(d)}} \Z[\beta]_{*-d+1}\otimes k(\eta)^\times\ar[r]_-{div_{CK}}&\sC\sK_{2*,*}^{\prime (1)}(X)
}
\end{equation}
As $\Div_{\sC\sK}:\Z[\beta]\otimes k(\eta)^\times\to \Omega_*^{(1)}(X)_{CK}$ is just the map formed by applying the functor $(-)\otimes_{\L}\Z[\beta]$ to $\Div: \L\otimes k(\eta)^\times\to\Omega_*^{(1)}(X)$, the desired commutativity follows from the commutativity of \eqref{eqn:HM3}.
\end{proof}

\begin{cor}\label{cor:GThyIso} 
Take $X\in\Sch/k$, and let $d$ be an integer with $d\ge \dim_kX$. Then 
\[
\gamma_n^X\circ \theta_{\sC\sK}(X):\Omega^{CK}_n(X)\to G_0(X)
\]
is an isomorphism for $n=d$ and an injection for $n=d-1$.
\end{cor}

\begin{proof} This follows from theorem~\ref{thm:ClassIso}  and corollary~\ref{cor:Injectivity}.
\end{proof}

\section{Lci pull-backs and fundamental classes} Theorem~\ref{thm:ClassIso} gives us the isomorphism
\[
\theta_{\sC\sK}:\Omega^{CK}_*(X)\to \sC\sK'_{2*,*}
\]
compatible with projective push forward, pull-back by open immersions, external products and 1st Chern class operators. However, $\Omega^{CK}_*$ is an oriented Borel-Moore homology theory, hence has in addition to these structures pull-back maps for arbitrary \lci-morphisms. The isomorphism $\theta_{\sC\sK}$ thus  endows the homology theory $\sC\sK'_{2*,*}$ with  \lci-pull-backs. The theory $X\mapsto \sK'_{2n,n}(X)=G_0(X)$ has \lci-pull-backs as well; we proceed to compare the two. As the \lci-pull-backs in $\Omega_*$ are defined using the classes of simple normal crossing divisors, we first need a result about these classes.

Fix an oriented cohomology theory $A^*$ with formal group law $F=F_A$. As usual, we use the notation $u+_Fv$ for $F(u,v)$ and more generally write $[n_1]_Fu_1+_F\ldots+_F[n_r]_Fu_r$ for the power series $F_{n_1,\ldots, n_r}(u_1,\ldots, u_r)$ that expresses the evident sum operation.
We may write $F_{n_1,\ldots, n_r}(u_1,\ldots, u_r)$ in the following form
\[
F_{n_1,\ldots, n_r}(u_1,\ldots, u_r)=\sum_{I\subset\{1,\ldots, r\}, I\neq\0} G^{n_1,\ldots, n_r}_I(u)\cdot u_I
\]
with the $G^{n_1,\ldots, n_r}_I(u)\in A_*(k)[[u_1,\ldots, u_r]]$ and $u_I=\prod_{i\in I}u_i$ If we assume (which we will) that $G^{n_1,\ldots, n_r}_I(u)$ is not divisible by any $u_j$ with $j\not\in I$, then the expression is unique (for instance,  $G^{n_1,\ldots, n_r}_{\{j\}}(u)=n_j$).

Now let $D=\sum_{i=1}^rn_iD_i$ be an effective simple normal crossing divisor on some $Y\in\Sm/k$, with each $D_i$ smooth and irreducible. As usual, for $I\subset\{1,\ldots, r\}$, let $D_I:=\cap_{i\in I}D_i$ and let $i_I:D_I\to |D|$, $i:|D|\to Y$ be the inclusions. Let $L_j=i^*\sO_Y(D_j)$; for 
$I=\{i_1,\ldots, i_m\}$, let $\tilde{c}_1(L_I)$ stand for the $m$-tuple $(\tilde{c}_1(L_{i_1}), \ldots, \tilde{c}_1(L_{i_m})$.  Suppose $Y$ has dimension $n$. The divisor class of $D$, $[D\to |D|]^A\in A_{n-1}(|D|)$ is defined as
\[
[D\to |D|]^A:=\sum_{I\subset\{1,\ldots, r\}, I\neq\0}i_{I*}G^{n_1,\ldots, n_r}_I(\tilde{c}_1(L_I))(1_{D_I})
\]

\begin{lem}\label{lem:DivClass} For $A_*=G_0[\beta,\beta^{-1}]$, we have 
\[
[D\to |D|]^{G_0}=[\sO_D] 
\]
in $G_0(|D|)$.
\end{lem}

\begin{proof} We write $[D\to |D|]$ for $[D\to |D|]^{G_0}$. and proceed by induction on $\sum_in_i$, where $D=\sum_{i=1}^rn_iD_i$. For $D$ a smooth irreducible divisor, $[D\to |D|]=1_D=[\sO_D]$, which takes care of the case $\sum_in_i=1$.  

For the general case,  let $D'=D-D_r=\sum_{i=1}^{r-1}n_iD_i+(n_r-1)D_r$, and let $j:|D'|\to |D|$, $i_r:D_r\to |D|$, $i: |D|\to Y$ be the inclusions.

The formal group law for $G_0[\beta,\beta^{-1}]$ is given by $F(u,v)=u+v-\beta uv\in\Z[\beta,\beta^{-1}][[u,v]]$. From this, an elementary computation yields
\begin{equation}\label{eqn:DivGpLaw}
[D\to |D|]= j_*[D'\to |D'|]+ i_{r*}[D_r\to|D_r|]-\beta\cdot \tilde{c}_1(k^*\sO_Y(D_r))( j_*[D'\to |D'])).
\end{equation}
Let $\sI_{D}$, $\sI_{D'}$, $\sI_{D_r}$ be the respective ideal sheaves in $\sO_Y$. We have the exact sequence of $\sO_D$-modules
\[
0\to \sI_{D_r}/\sI_{D}\to \sO_D\to \sO_{D_r}\to0
\]
Furthermore, the multiplication map defines an isomorphism of $\sO_{D'}$-modules
\[
\sO_{D'}\otimes_{\sO_Y}\sI_{D_r}\cong  \sI_{D_r}/\sI_{D}.
\]
This gives us the identities in $G_0(D)[\beta,\beta^{-1}]$:
\begin{align*}
[\sO_D]&=\sO_Y(-D_r)\cdot j_*[\sO_{D'}]+i_{r*}[\sO_{D_r}]\\
&= j_*[\sO_{D'}]+i_{r*}[\sO_{D_r}]-(1-\sO_Y(-D_r))\cdot j_*[\sO_{D'}]\\
&= j_*[\sO_{D'}]+i_{r*}[\sO_{D_r}]-\beta\tilde{c}_1(k^*\sO_Y(D_r))(j_*[\sO_{D'}]),
\end{align*}
the last identity following from 
\[
\tilde{c}_1(L)(\alpha)=\beta^{-1}(1-L^{-1})\cdot\alpha
\]
for $L$ a line bundle on $|D|$ and $\alpha\in G_0(|D|)[\beta,\beta^{-1}]$. Combined with \eqref{eqn:DivGpLaw} and our induction hypothesis, this proves the lemma.
\end{proof}

\begin{prop}\label{prop:LciCompat} Let $f:Y\to X$ be an \lci-morphism of relative dimension $d$ in  $\Sch/k$. Then the diagram
\[
\xymatrix{
\Omega^{CK}_n(X)\ar[r]^{\gamma\circ\theta}\ar[d]_{f^*}&G_0(X)\ar[d]^{f^*}\\
\Omega^{CK}_{n+d}(Y)\ar[r]_{\gamma\circ\theta}&G_0(Y)
}
\]
commutes.
\end{prop}

\begin{proof} It suffices to consider two cases: (a) $f$ a regular embedding, (b) $f$ a smooth morphism.

We first note that $\Omega^{CK}_n(X)$ is generated by the classes $[g:Z\to X]$ with $Z\in\Sm/k$ irreducible of dimension $n$ over $k$, and $g$ a projective morphism; here $[g:Z\to X]$ stands for the element $g_*(1_Z)$, with $1_Z$ the unit in the graded ring $\Omega^{CK}_*(Z)$. The image of $[g:Z\to X]$ is given by the same formula, that is 
\[
(\gamma\circ\theta)([g:Z\to X])=Rg_*(\sO_Z)
\]
where $Rg_*:G_0(Z)\to G_0(X)$ is the usual projective push-forward map on $G_0$: $Rg_*([\sF]):=\sum_i(-1)^i[R^ig_*\sF]$. We now compute in cases (a) and (b):\\
\\
{\em Case (b):} $f:Y\to X$ smooth. Then $f^*[g:Z\to X]=[p_2:Z\times_XY\to Y]$ and 
\begin{align*}
Rp_{2*}[\sO_{Z\times_XY}]&=Rp_{2*}(p_1^*[\sO_Z])\\
&=f^*(Rg_*[\sO_Z])
\end{align*}
the last identity being the base-change isomorphism for the flat morphism $f$. \\
\\
{\em Case (a):} $f:Y\to X$ a regular embedding, say of codimension $c$. Since $X$ is quasi-projective, we can factor $f$ through a sequence of regular codimension one embeddings
\[
Y=X_c\hookrightarrow X_{c-1}\hookrightarrow\ldots\hookrightarrow X_1\hookrightarrow X;
\]
this reduces us to the case $c=1$, that is, $f:Y\to X$ is a Cartier divisor on $X$. In this case, $\Omega^{CK}_*(X)$ is generated by three types of elements (with $g, Z$ as before):
\begin{enumerate}
\item $g:Z\to X$ with $g(Z)\subset Y$.
\item $g:Z\to X$ such that $g^{-1}(Y)$ is a simple normal crossing divisor on $Z$
\item $g^{-1}(Y)=\0$.
\end{enumerate}
In case (3), $f^*([g:Z\to X])=0$ by definition, and clearly $f^*(Rg_*(\sO_Z))$ is zero as well. In case (1), let $L=f^*(\sO_X(Y))$ and let $\bar{g}:Z\to Y$ be the unique map with $f\circ\bar{g}=g$. Then by definition
\[
f^*([g:Z\to X]):=\tilde{c}_1(L)([\bar{g}:Z\to Y]).
\]
On the $G$-theory side, let $\sF$ be a coherent sheaf on $Y$. Then $Rf_*[\sF]=f_*\sF$, since $f$ is finite, and $f^*([f_*\sF]):=[\sF]-[\Tor^{\sO_X}_1(f_*\sF,\sO_Y)]$. Via the exact sequence
\[
0\to \sO_X(-Y)\to \sO_X\to \sO_Y\to0
\]
we see that 
\[
[\sF]-[\Tor^{\sO_X}_1(f_*\sF,\sO_Y)]=(1-L^{-1})\cdot[\sF]
\]
in $G_0(Y)$, so $f^*([f_*\sF])=(1-L^{-1})\cdot[\sF]$. Applying this to the coherent sheaves $R^i\bar{g}_*\sO_Z$ and noting that $\tilde{c}_1^{G_0}(L)$ is multiplication by $(1-L^{-1})$ shows that 
\[
(\gamma\circ\theta)(f^*[g:Z\to X])=f^*(Rg_*(\sO_Z))
\]
as desired. 

Finally, in case (2), the diagram
\[
\xymatrix{
g^{-1}(Y)\ar[r]^{\bar{f}}\ar[d]_{\bar{g}}&Z\ar[d]^g\\
Y\ar[r]_f&X}
\]
is cartesian and Tor-independent, hence we have the base-change identity
\[
f^*Rg_*[\sO_Z]=R\bar{g}_*(\bar{f}^*[\sO_Z])=R\bar{g}_*([\sO_{g^{-1}(Y)}]).
\]
Thus it suffices to show that 
\begin{equation}\label{eqn:DivId}
[\sO_{g^{-1}(Y)}]=(\gamma\circ\theta)(\bar{f}^*(1_Z))
\end{equation}
in $G_0(g^{-1}(Y))$. But by definition, we have $\bar{f}^*(1_Z)=[g^{-1}(Y)\to |g^{-1}(Y)|]$, so the identity \eqref{eqn:DivId} follows from lemma~\ref{lem:DivClass}.
\end{proof}

\begin{definition} Let $X\in\Sch/k$ have pure dimension $d$. The {\em fundamental class} $[X]_{CK}\in \Omega^{CK}_d(X)$ is the element corresponding to $[\sO_X]\in G_0(X)$ under the isomorphism $\Omega_d^{CK}(X)\to G_0(X)$.
\end{definition}

\begin{thm} Let $f:Y\to X$ be an \lci-morphism in $\Sch/k$. Assume that $X$ has pure dimension $d_X$ and $Y$ has pure dimension $d_Y$. then $f^*([X]_{CK})=[Y]_{CK}$ in $\Omega_{d_Y}^{CK}(Y)$.
\end{thm}

\begin{proof} If $f$ is smooth, then $f^*[\sO_X]=[\sO_Y]$ in $G_0(Y)$. If $f$ is a regular embedding, then similarly $f^*[\sO_X]=[\sO_Y]$ in $G_0(Y)$. Thus, for an arbitrary \lci-morphism $f$, $f^*[\sO_X]=[\sO_Y]$ in $G_0(Y)$. The theorem thus follows from the definition of the fundamental class, corollary~\ref{cor:GThyIso}  and  proposition~\ref{prop:LciCompat}.
\end{proof} 

Let $\Omega_*^\times:=\Omega_*\otimes_\L\Z[\beta,\beta^{-1}]$ and $\Omega_*^+:=\Omega_*\otimes_\L\Z$, the first tensor product defined via the classifying map for the formal group law $(u+v-\beta uv,
Z[\beta, \beta^{-1}]$, the second for the additive formal group law $(u+v, \Z)$.  Dai's theorem \cite{Dai} states that the classifying map $\Omega_*^\times\to G_0[\beta,\beta^{-1}]$ is an isomorphism, while it is shown in \cite[theorem 4.5.1]{LevineMorel} that the classifying map $\Omega_*^+\to \CH_*$ is an isomorphism. 

In particular,  inverting $\beta$ induces the map
\[
\vartheta_\times:\Omega_*^{CK}\to \Omega_*^\times,
\]
which after the isomorphisms $\Omega_*^\times\cong G_0[\beta,\beta^{-1}]\cong \sK'_{2*,*}$ and 
$\Omega_*^{CK}\cong \sC\sK'_{2*,*}$ is just the canonical map $\rho: \sC\sK'_{2*,*}\to  \sK'_{2*,*}$. Similarly, taking the quotient by $\beta$ defines the map
\[
\vartheta_+:\Omega_*^{CK}\to \Omega_*^+
\]
which is the same as the classifying map for the additive theory $\CH_*$. 

Both $G_0[\beta,\beta^{-1}]$ and $\CH_*$ admit fundamental classes for equi-dimensional finite type $k$-schemes, functorial with respect to pull-back by \lci morphisms. For $X\in\Sch/k$ of dimension $d$ over $k$, the fundamental class $[X]_G\in G_0(X)[\beta,\beta^{-1}]_d$ is $\beta^d[\sO_X]$, where $[\sO_X]\in G_0(X)$ is the class of the structure sheaf. For $\CH_*$, $[X]_\CH$ is the cycle class of the scheme $X$; concretely, this is $\sum_{i=1}^rn_i[X_i]$ where $X$ has irreducible components $X_1,\ldots, X_r$ and $n_i$ is the length of the local ring $\sO_{X_i,\eta_i}$ at the generic point $\eta_i\in X_i$.

\begin{prop} Let $X$ be an equi-dimensional finite type $k$-scheme. Then after canonical identifications $\Omega_*^\times\cong G_0[\beta, \beta^{-1}]$, $\Omega_*^+\cong \CH_*$, we have
\[
\vartheta_\times([X]_{CK})=[X]_G;\quad \vartheta_+([X]_{CK})=[X]_{CH}.
\]
\end{prop}

\begin{proof}  Suppose $X$ has dimension $d$ over $k$. For $G$-theory, it follows by the definition of $[X]_{CK}$ that the corresponding element $[X]_{CK}\in \sC\sK'_{2d,d}(X)$ maps to $[\sO_X]\in \sK'_{2d,d}(X)=G_0(X)$ under the canonical map $\rho: \sC\sK'\to  \sK'$. This proves the result for $G$-theory.

For $\CH_*$, the fundamental class $[X]_\CH$ is determined by its restriction to each generic point of $X$, so we may assume that $X$ is irreducible and $X_\red$ is smooth over $k$. Write $[X]_\CH=n[X_\red]_\CH$. Shrinking $X$ if necessary, we may assume that $[\sO_X]=n[\sO_{X_\red}]$ in $G_0(X)$, and thus $[X]_{CK}=n[X_\red]_{CK}$ in $\Omega_d^{CK}(X)=\Omega_d^{CK}(X_\red)$. This reduces us to the case of $X\in \Sm/k$. But then we already have the fundamental class of $X$ in any oriented cohomology theory $A$, given by 
\[
[X]_A:=p^*(1)\in A^0(X)
\]
where $p:X\to\Spec k$ is the structure morphism and $1\in A^0(k)$ is the unit. Since every natural transformation of oriented cohomology theories preserves the unit and is compatible with pull-back for arbitrary morphisms in $\Sm/k$, we have $\vartheta_+([X]_{CK})=[X]_{CH}$, as desired.
\end{proof}


\begin{thebibliography}{99}

\bibitem{Cai}
Cai, S., {\em Algebraic connective $K$-theory and the niveau filtration}.  J. Pure Appl. Algebra {\bf 212}, No. 7, 1695-1715 (2008).

\bibitem{Dai}
Dai, S. {\it Algebraic Cobordism and Grothendieck group over singular schemes}, Homology, Homotopy and Applications, vol. 12(1), 2010, pp.93-110. 

\bibitem{Motivic}
Dundas, B. I., Levine, M., {\O}stv{\ae}r, P. A., R\"ondigs, O. and Voevodsky, V. {\bf
Motivic homotopy theory}. 
Lectures from the Summer School held in Nordfjordeid, August 2002. Universitext. Springer-Verlag, Berlin, 2007

\bibitem{GoerssJardine}
Goerss, P.\,G.\,and Jardine, J.\,F., {\sl Localization theories for simplicial presheaves}. Canad. J. Math. {\bf 50} (1998), no. 5, 1048--1089.

\bibitem{GRSO}
Guti\'errez, J.\,J., R\"ondigs, O., Spitzweck, M., Paul Arne {\O}stv{\ae}r, P.\,A., {\em Motivic slices and colored operads}, preprint 2010, rev. 2012, 35 pages  \verb! arXiv:1012.3301v2 [math.AG]!

\bibitem{HopkinsMorel}
M. Hopkins, F. Morel, {\em Slices of MGL}. lecture (Hopkins), Harvard Univ. Dec. 2, 2004. Available as ``Week 8" on 
a webpage of T. Lawson presenting notes from  a seminar given by Mike Hopkins, Harvard, fall of 2004. 
\verb!http://www.math.umn.edu/~tlawson/motivic.html!

\bibitem{Hoyois}
M. Hoyois, {\em From algebraic cobordism to motivic cohomology}. Preprint 2012, 30 pages.\\
\verb!http://math.northwestern.edu/~hoyois/papers/hopkinsmorel.pdf!

\bibitem{Jardine2}
Jardine, J.F., {\sl Motivic symmetric spectra}, {\em Doc. Math.} 5 (2000), 445--553.

\bibitem{LevineHC}
Levine, M.,  {\sl The homotopy coniveau tower}. J Topology {\bf 1} (2008) 217--267.

\bibitem{LevineOrient}
Levine, M., {\sl Oriented cohomology, Borel-Moore homology and algebraic cobordism}. Michigan Math. J., Special Issue in honor of Melvin Hochster, Volume 57, August 2008, 523-572

\bibitem{LevineComparison} 
Levine, M.  {\sl Comparison of cobordism theories}. J. Algebra {\bf 322} (2009), no. 9, 3291--3317.

\bibitem{LevineFundClass}
Levine, M.,  {\sl Fundamental classes in algebraic cobordism}. $K$-Theory
{\bf 30} (2) (2003), 129-135. 

\bibitem{TechLoc}
Levine, M., {\sl Techniques of localization in the theory of algebraic cycles}, 
J. Alg. Geom. {\bf 10} (2001) 299-363.

\bibitem{LevineAlgKthyI}
Levine, M., {\sl Algebraic $K$-theory of schemes, I}, preprint (2004)\\ \verb!http://www.uni-due.de/~bm0032/publ/KthyMotI12.01.pdf!

\bibitem{LevineMorel} Levine, M. and Morel, F. {\bf Algebraic cobordism}. Springer Monographs in Mathematics. Springer, Berlin, 2007.

\bibitem{Mocanasu}
Mocanasu, M. {\sl Borel-Moore functors and algebraic oriented theories}. Preprint (2004). $K$-theory preprint archive 713, \verb!http://www.math.uiuc.edu/K-theory/0713/!

\bibitem{MorelLec} 
Morel, F. {\sl  An introduction to $\mathbb A\sp 1$-homotopy theory}.  {\em Contemporary developments in algebraic $K$-theory} 357--441, ICTP Lect. Notes, XV, Abdus Salam Int. Cent. Theoret. Phys., Trieste, 2004.

\bibitem{MorelVoev}  Morel, F. and Voevodsky, V., {\sl $\A^1$-homotopy theory of schemes},
Inst. Hautes \'Etudes Sci. Publ. Math. { 90} (1999), 45--143.

\bibitem{Panin2} Panin, I., {\sl Push-forwards in oriented cohomology theories of algebraic varieties II}. Preprint 2003. \verb!http://www.math.uiuc.edu/K-theory/0619/!

\bibitem{PaninPimenovRoendigs} Panin, I., Pimenov, K. and R\"ondigs, O.
{\sl A universality theorem for Voevodsky's algebraic cobordism spectrum}. Preprint (2007). $K$-theory preprint archive \verb!http://www.math.uiuc.edu/K-theory/0846/!


\bibitem{PaninPimenovRoendigs2}
 Panin, I., Pimenov, K. and R\"ondigs, O., {\em On Voevodsky's algebraic $K$-theory spectrum}. Algebraic topology, 279--330, Abel Symp., 4, Springer, Berlin, 2009.
 
\bibitem{PaninPimenovRoendigs3}
 Panin, I., Pimenov, K. and R\"ondigs, O., {\em On the relation of Voevodsky's algebraic cobordism to Quillen's $K$-theory}. Invent. Math. 175 (2009), no. 2, 435Ð451.

\bibitem{PelaezFunctSlice}
Pelaez, P. {\sl On the functoriality of the slice filtration}, preprint 2010, rev. 2012
\verb!arXiv:1002.0317!

\bibitem{QuillenKThyI}
 Quillen, D., {\em Higher algebraic K-theory. I}. Algebraic K-theory, I: Higher K-theories (Proc. Conf., Battelle Memorial Inst., Seattle, Wash., 1972), pp. 85Ð147. Lecture Notes in Math., Vol. 341, Springer, Berlin 1973.
 
 
\bibitem{RSO}
R\"ondigs, O., Spitzweck, M.,  {\O}stv{\ae}r, P.\,A., {\em Motivic strict ring models for $K$-theory}. Proc. Amer. Math. Soc. 138 (2010), no. 10, 3509--3520. 

\bibitem{VoevSlice1} 
Voevodsky, V., {\em Open problems in the motivic stable homotopy theory. I'}. Motives, polylogarithms and Hodge theory, Part I (Irvine, CA, 1998), 3--34, Int. Press Lect. Ser., {\bf 3, I} Int. Press, Somerville, MA, 2002.

\bibitem{VoevSlice1bis}  
Voevodsky, V., {\em A possible new approach to the motivic spectral sequence for algebraic $K$-theory}.   Recent progress in homotopy theory (Baltimore, MD, 2000) 371--379, Contemp. Math., {\bf 293} Amer. Math. Soc., Providence, RI, 2002. 

\bibitem{VoevSlice2}
Voevodsky, V., {\em On the zero slice of the sphere spectrum},   Proc. Steklov Inst. Math.  no. 3 {\bf 246} (2004) 93--102.

\bibitem{Voevodsky}
Voevodsky, V., {\sl   $A^1$-homotopy theory}, Proceedings of the International 
Congress of Mathematicians, Vol. I (Berlin, 1998). Doc. Math. 1998, 
Extra Vol. I, 579--604.

\end{thebibliography}
\end{document}